\documentclass[a4paper,reqno,11pt]{amsart}

\usepackage[T1]{fontenc}
\usepackage[utf8]{inputenc}
\usepackage{xspace,
		amsfonts}
\usepackage{amsmath}
\usepackage{amssymb}
\usepackage{graphicx,
		bbm,
		tikz,
		subfig}
  \usepackage{xfrac}
		\usepackage{dsfont}
\usepackage{amsthm}
\usepackage{mathtools}
\usepackage{enumitem}
\usepackage{verbatim}
\usepackage[autostyle]{csquotes}

\usepackage{cite}

\mathtoolsset{showonlyrefs}

\usepackage{geometry}
\newtheorem{theorem}{Theorem}[section]
\newtheorem{lemma}[theorem]{Lemma}

\theoremstyle{remark}
\newtheorem{remark}[theorem]{Remark}
\theoremstyle{definition}

\newcommand{\R}{\mathbb{R}}

\newcommand{\EE}{\mathcal{E}}
\newcommand{\dx}{\,dx}
\newcommand{\ds}{\,ds}

\tikzstyle{nodino}=[circle,draw,fill,inner sep=0pt,minimum size=0.5mm]
\tikzstyle{infinito}=[circle,inner sep=0pt,minimum size=0mm]
\tikzstyle{nodo}=[circle,draw,fill,inner sep=0pt, minimum size=0.5*width("k")]
\tikzstyle{nodo_vuoto}=[circle,draw,inner sep=0pt, minimum size=0.5*width("k")]
\usetikzlibrary{graphs}
\tikzset{every loop/.style={min distance=10mm,in=300,out=240,looseness=10}}
\tikzset{place/.style={circle,thick,draw=blue!75,fill=blue!20,minimum
		size=6mm}}
\tikzset{place2/.style={circle,thick,draw=red!75,fill=red!20,minimum
		size=6mm}}

\title[1D defocusing NLSE with nonlinear $\delta$-interactions]{Normalized solutions of one-dimensional defocusing NLS equations with nonlinear point interactions}

\author[ ]{Daniele Barbera}
\address[D. Barbera]{Politecnico di Torino, Dipartimento di Scienze Matematiche ``G.L. Lagrange'' Corso Duca degli Abruzzi 24, 10129, Torino, Italy.}
\email{daniele.barbera@polito.it}

\author[ ]{Filippo Boni}
\address[F. Boni]{Politecnico di Torino, Dipartimento di Scienze Matematiche ``G.L. Lagrange'' Corso Duca degli Abruzzi 24, 10129, Torino, Italy.}
\email{filippo.boni@polito.it}

\author[ ]{Simone Dovetta}
\address[S. Dovetta]{Politecnico di Torino, Dipartimento di Scienze Matematiche ``G.L. Lagrange'', Corso Duca degli Abruzzi 24, 10129, Torino, Italy.}
\email{simone.dovetta@polito.it}

\author[ ]{Lorenzo Tentarelli}
\address[L. Tentarelli]{Politecnico di Torino, Dipartimento di Scienze Matematiche ``G.L. Lagrange'', Corso Duca degli Abruzzi 24, 10129, Torino, Italy.}
\email{lorenzo.tentarelli@polito.it}

\begin{document}
	
	\begin{abstract}
		We investigate normalized solutions for doubly nonlinear Schr\"odinger equations on the real line with a defocusing standard nonlinearity and a focusing nonlinear point interaction of $\delta$--type at the origin.  We provide a complete characterization of existence and uniqueness for normalized solutions and for energy ground states for every value of the nonlinearity powers. We show that the interplay between a defocusing standard and a focusing point nonlinearity gives rise to new phenomena with respect to those observed with single nonlinearities, standard combined nonlinearities, and combined focusing standard and pointwise nonlinearities.
	\end{abstract}
	
	\maketitle

	\vspace{-.5cm}
	
	
	\section{Introduction}
	
The present paper investigates normalized solutions of doubly nonlinear Schr\"odinger equations on the real line with a defocusing standard nonlinearity and a focusing point nonlinearity of $\delta$--type located at the origin, namely solutions of the problem
\begin{equation}
	\label{norm-nlse}
	\begin{cases}
	u''-|u|^{p-2}u+|u|^{q-2}\delta_0u=\lambda u & \text{ on }\R\\
	\|u\|_{L^2(\R)}^2=\mu\,, &
	\end{cases}
\end{equation}
where $\lambda\in\R$, $\mu>0$, and $2<p,q<\infty$. Equivalently, \eqref{norm-nlse} can be rewritten as the following nonlinear boundary value problem
\begin{equation}
	\label{nlse}
	\begin{cases}
	u''-|u|^{p-2}u=\lambda u & \text{on }\R\setminus\{0\}\\
	u'(0^-)-u'(0^+)=|u(0)|^{q-2}u(0) & \\
	\|u\|_{L^2(\R)}^2=\mu\,, &
	\end{cases}
\end{equation}
where $u'(0^-),u'(0^+)$ denote the left and right derivative of $u$ at the origin, respectively, and the nonlinear boundary condition in the second line encodes the point nonlinearity of $\delta$--type.

\smallskip

Even though the appearance of point perturbations of Schr\"odinger operators can be traced back to Fermi's work \cite{F36} in 1936, it is perhaps since the last decade of the previous century that nonlinear Schr\"odinger equations involving point interactions of $\delta$--type started gathering a prominent interest. Exploiting the presence of such singular point potentials to describe, e.g., strongly--localized defects in a medium or confinement effects in small spatial regions, models of this kind have been proposed by now for a wide range of phenomena in solid state and condensed matter physics (see for instance \cite{DM11,JLPS95,MB02,N98,SCMS20,SKBRC01} and references therein). From the technical point of view, delta--type terms have been implemented in nonlinear models mainly in two ways: either coupling the linear Schr\"odinger equation with a point nonlinearity of $\delta$--type $|u|^{q-2}\delta_0 u$ (see \cite{T23} for a comprehensive review on this approach and, in particular, \cite{ACCT20,ACCT21,ANO13,BD21,CFT19} for a specific focus on standing waves), or perturbing nonlinear Schr\"odinger equations involving standard nonlinearities by a linear delta--interaction $\delta_0 u$ (see, e.g., \cite{ABCT_cvpde,ABCT_jmp,AN09,AN13,ANV13,BC23,CFN21,CFN23,DSS25,FGI22,FJ08,FN23,FOO08,GHW04,HMZ07,KO09,LFFKS08}). Conversely, up to the last few years doubly nonlinear Schr\"odinger equations have been extensively studied with standard nonlinearities only (we refer for instance to \cite{CKS23,CS21,DT19,FH21,KOPV17,KO09,LCMR16,LRN20,M08,MZZ17,O95,PS22,ST_jde,ST_na,S20a,S20b,TVZ07,T16}). However, the recent works \cite{ABD22,BD21,BD22} started the investigation of doubly nonlinear models combining standard and point nonlinearities of $\delta$--type as in \eqref{norm-nlse} above. Precisely, these papers began the study of normalized solutions of focusing--focusing doubly nonlinear equations, i.e.
\[
u''+|u|^{p-2}u+|u|^{q-2}\delta_0u=\lambda u,
\] 
both on the real line and on metric graphs, where existence of fixed mass solutions is shown to be governed by a nontrivial interplay between the two nonlinear terms.

\smallskip
The present paper pushes forward the analysis of doubly nonlinear models with point nonlinearities of $\delta$--type, considering the case of one-dimensional Schr\"odinger equations \eqref{norm-nlse} with a defocusing standard nonlinearity combined with a focusing point nonlinearity of $\delta$--type based at the origin.  Observe that, since the standard nonlinearity in \eqref{nlse} is defocusing, it is well known that, without the point perturbation, there is no function in $L^2(\R)$ satisfying the first line of \eqref{nlse} on the whole $\R$. Hence, we aim at understanding whether an attractive nonlinear point perturbation at the origin can restore the existence of normalized solutions, and how this depends on the specific values of the two nonlinearity powers.  Even though combinations of defocusing-focusing standard nonlinearities have been studied for instance in \cite{CKS23,CS21,LRN20}, when turning to $\delta$--type interactions we are only aware of the investigations in \cite{KO09,DSS25}, where a linear point term (that is, $q=2$ in \eqref{norm-nlse}) is coupled with a defocusing standard nonlinearity on the real line \cite{KO09} and on metric graphs \cite{DSS25}. Note, furthermore, that both \cite{KO09} and \cite{DSS25} are mainly focused on the study of solutions of the equation for fixed $\lambda$ without the mass constraint, and only partial information is provided on normalized solutions in \cite{DSS25}. Hence, to the best of our knowledge, this is the first paper considering attractive nonlinear  perturbations of $\delta$--type of standard defocusing nonlinear Schr\"odinger equations.

Let us now state our main results, with which we develop a comprehensive description of problem \eqref{norm-nlse} for every values of $p,q>2$.  First, we provide a complete characterization of the set of (weak) solutions in $H^1(\R)$ of the doubly nonlinear problem
\begin{equation}
	\label{nls-nomass}
	\begin{cases}
	u''-|u|^{p-2}u=\lambda u & \text{on }\R\setminus\{0\}\\
	u'(0^-)-u'(0^+)=|u(0)|^{q-2}u(0) & 
	\end{cases}
\end{equation}
for every fixed value of $\lambda\in\R$, i.e. \eqref{nlse} without the mass constraint.
\begin{theorem}
	\label{thm:nomass}
	Let $p,q>2$. For every $\lambda\in\R$, if $u\in H^1(\R)$ is a nontrivial solution of \eqref{nls-nomass}, then $u$ is even and (up to a change of sign) positive and radially decreasing on $\R$. Moreover,

	\begin{enumerate}[label=(\roman*)]
		\item when $\lambda<0$, \eqref{nls-nomass} has no nontrivial solution in $H^1(\R)$;
		\smallskip
		
		\item when $\lambda=0$, \eqref{nls-nomass} has a positive solution in $H^1(\R)$ if and only if $p\in(2,6)$ and $ q\neq\frac p2+1$, and this solution is unique;
		
		\smallskip
		\item when $\lambda>0$,
		
		\smallskip
		\begin{enumerate}[label=(\alph*)]
			\item if $ q>\frac p2+1$, \eqref{nls-nomass} has a unique positive solution in $H^1(\R)$;
			
			\smallskip
			\item if $ q=\frac p2+1$, \eqref{nls-nomass} has a positive solution in $H^1(\R)$ if and only if $p>8$, and this solution is unique;
			
			\smallskip
			\item if $ q<\frac p2+1$, there exists $\overline{\lambda}_{p,q}>0$ such that \eqref{nls-nomass} has no nontrivial solution in $H^1(\R)$ when $\lambda>\overline{\lambda}_{p,q}$, it has a unique positive solution in $H^1(\R)$ when $\lambda=\overline{\lambda}_{p,q}$, and it has exactly two positive solutions in $H^1(\R)$ when $\lambda\in(0,\overline{\lambda}_{p,q})$.
		\end{enumerate}
	\end{enumerate}
\end{theorem}
Theorem \ref{thm:nomass} spots for the first time two elements that will be crucial all along our discussion: the value $p=6$, as a critical threshold for the problem with $\lambda=0$, and the identity $ q=\frac p2+1$. In particular, we highlight that, even though this latter relation between $p$ and $q$ was already proved to be relevant for the focusing--focusing doubly nonlinear model on metric graphs in \cite{ABD22, BD22}, it was shown to play no role on the real line when both nonlinearities are focusing (at least in the case $2<p,q<6$, see \cite{BD21}). Here, on the contrary, the actual location of the point $(p,q)$ with respect to the line $q=\frac p2+1$ in the $pq$--plane determines a sharp transition for positive solutions of \eqref{nls-nomass}, that exist and are unique for every $\lambda>0$ when $(p,q)$ is above this line (Theorem \ref{thm:nomass}(iii-a)), whereas they exist only for values of $\lambda$ smaller than a positive threshold for $(p,q)$ below it (Theorem \ref{thm:nomass}(iii-c)), and uniqueness fails in this regime. We also highlight that, in view of \cite[Theorem 1]{KO09}, the lack of uniqueness identified in Theorem \ref{thm:nomass}(iii-c) is a purely doubly nonlinear effect. Indeed, \cite{KO09} considers the problem \eqref{nls-nomass} with a linear delta--potential ($q=2$) multiplied by a real parameter $\gamma$. For attractive interaction $\gamma>0$, it is proved therein that the problem admits a positive solution in $H^1(\R)$ if and only if $\lambda\in[0,\gamma^2/4)$ and $p\in(2,6)$ or $\lambda\in(0,\gamma^2/4)$ and $p>6$, but at fixed $\lambda$ this solution is always unique. This is in sharp contrast with Theorem \ref{thm:nomass}(iii-c), which ensures that, for every $p>2$ and $q$ sufficiently close to $2$ (depending on $p$), problem \eqref{nls-nomass} admits multiple positive solutions for a whole interval of values of $\lambda$. Observe also that, even though the existence of positive solutions occuring only for $\lambda$ below a threshold is independent of the $\delta$--interaction being linear or nonlinear, in the linear case the threshold is $\gamma^2/4$, that does not depend on $p$, whereas in our doubly nonlinear model the value $\overline{\lambda}_{p,q}$ is sensitive to both nonlinearities.

We then turn our attention to normalized solutions. Our first aim is to provide a complete portrait for existence of such solutions depending on the value of $\mu$. From this perspective, when in the following we say that problem \eqref{nlse} admits a normalized solution with mass $\mu>0$, we will mean that, given $\mu$, there exist $u\in H^1(\R)$ and $\lambda\geq0$ for which \eqref{nlse} is satisfied.

Since the behaviour of normalized solutions turns out to be quite varied depending on the values of $p$ and $q$, to ease the statement of the next theorem we introduce the following sets in the $pq$--plane (see Figure \ref{fig:piano_pq}):

\begin{itemize}
	
	\item $\displaystyle A = \left\{(p,q)\in\R^2\,:\, 2<p<6,\, 2<q<\frac p2+1\right\}$;
	
	\smallskip
	\item $\displaystyle B = \left\{(p,q)\in\R^2\,:\, p\geq6,\, 2<q<4\right\}$;
	
	\smallskip
	\item $\displaystyle C=\left\{(p,q)\in\R^2\,:\,p>6,\, 4< q<\frac p2+1\right\}$
	
	\smallskip
	\item $\displaystyle D=\left\{(p,q)\in\R^2\,:\, p\geq6,\, q>\frac p2+1\right\}$;
	
	\smallskip
	\item $\displaystyle E=\left\{(p,q)\in\R^2\,:\,2<p<6,\, q>4\right\}$
	
	\smallskip
	\item $\displaystyle F=\left\{(p,q)\in\R^2\,:\, 2 < p< 6,\, \frac p2+1 < q <4 \right\}$;
	
	\smallskip
	\item $\displaystyle G=\left\{(p,q)\in\R^2\,:\,2<p<6,\, q=4\right\}$;
	
	\smallskip
	\item $\displaystyle H=\left\{(p,q)\in\R^2\,:\, p>6,\, q=4\right\}$;
	
	\smallskip
	\item $\displaystyle I=\left\{(p,q)\in\R^2\,:\,p>2,\,q=\frac p2+1\right\}$.
	
\end{itemize}
Observe that these sets form a partition of the region $(2,\infty)\times(2,\infty)$ in the $pq$--plane.
\begin{figure}[t]
	\centering
	\includegraphics[width=0.7\textwidth]{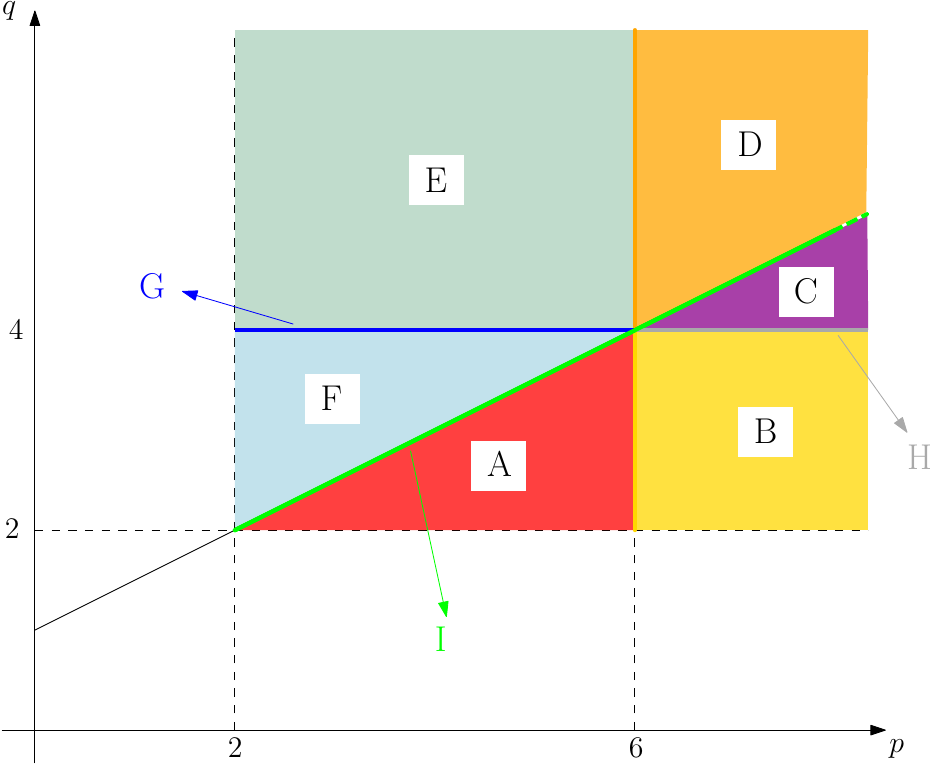}
	\caption{The subsets of the $pq$-plane identified by the straight lines $p=2$, $p=6$, $q=2$, $q=4$, and $q=\frac p2+1$.}
	\label{fig:piano_pq}
\end{figure}

\begin{theorem}
\label{thm:exsol}
Let $p,q>2$. The following holds:

\begin{enumerate}[label=(\roman*)]
	\item if $(p,q)\in A\cup E$, then there exists $\mu_{p,q}>0$ such that \eqref{nlse} admits a solution if and only if $\mu\leq\mu_{p,q}$;
	
	\smallskip
	\item if $(p,q)\in B\cup D$, then \eqref{nlse} admits a solution for every $\mu>0$;
	
	\smallskip
	\item if $(p,q)\in C\cup F$, then there exists $\mu_{p,q}>0$ such that \eqref{nlse} admits a solution if and only if $\mu\geq\mu_{p,q}$;
	
	\smallskip
	\item if $(p,q)\in G$, then there exists $\mu_{p,4}>2$ such that \eqref{nlse} admits a solution if and only if $\mu\in(2,\mu_{p,4}]$;
	
	\smallskip
	\item if $(p,q)\in H$, then \eqref{nlse} admits a solution if and only if $\mu>2$;
	
	\smallskip
	\item if $(p,q)\in I$, then \eqref{nlse} has no solution for every $\mu>0$ when $p\leq8$, whereas it admits a solution for every $\mu>0$ when $p>8$.
\end{enumerate}
Furthermore, if $p,q,\mu$ are as in (i), (ii), (iv), (v) and (vi), up to a change of sign the solution of \eqref{nlse} is unique.  
\end{theorem}
Theorem \ref{thm:exsol} identifies in the straight lines $p=6$, $q=4$, and $q=\frac p2+1$ the edges separating the regimes of nonlinearities where the behaviour of normalized solutions is sensibly different. On the one hand, the value $p=6$ governs the existence of large mass solutions. Indeed, for any fixed $q>2$, Theorem \ref{thm:exsol} shows that crossing $p=6$ (that is, passing from regions A to B, G to H, or E to D) restores existence of solutions with large $\mu$. This transition is rooted in the lack of nontrivial solutions in $H^1(\R)$ of \eqref{nls-nomass} with $\lambda=0$ for every $p\geq6$, as stated in Theorem \ref{thm:nomass}(ii). Indeed, as shown in Section \ref{sec:mass} below, for every $p<6$ the critical value $\mu_{p,q}$ in Theorem \ref{thm:exsol}(i)--(iv) is exactly the mass of the unique positive solution in $H^1(\R)$ of \eqref{nls-nomass} with $\lambda=0$ (computed in \eqref{eq-mu0}). On the other hand, the existence of solutions with small mass is determined by $q=4$ and $q=\frac p2+1$: in the regions enclosed between the two straight lines (C, F, G and H) no such solution exist, whereas they do exist in the outer regions (A, B, D and E).

Furthermore, we point out that the lack of a discussion on the uniqueness of normalized solutions when $p,q,\mu$ are as in Theorem \ref{thm:exsol}(iii) is not accidental. In fact, in this regime uniqueness fails to be true in general.
\begin{theorem}
	\label{thm:nonun}
	There exist $p,q,\mu$ as in Theorem \ref{thm:exsol}(iii) for which \eqref{nlse} admits two different solutions.
\end{theorem}
Recall that, as it is well-known, $u\in H^1(\R)$ solves \eqref{nlse} if and only if $u$ is a critical point of the associated energy functional $E_{p,q}:H^1(\R)\to\R$
\begin{equation}
	\label{E}
	E_{p,q}(v):=\frac12\|v'\|_{L^2(\R)}^2+\frac1p\|v\|_{L^p(\R)}^p-\frac1q|v(0)|^q
\end{equation}
constrained to the set of functions with mass $\mu$
\begin{equation}
\label{eq-vincolo}
H_\mu^1(\R):=\left\{v\in H^1(\R)\,:\,\|v\|_{L^2(\R)}^2=\mu\right\},
\end{equation}
with the parameter $\lambda$ in \eqref{nlse} popping up as a Lagrange multiplier associated to the mass constraint. In particular, among all critical points of the energy, a specific interest is usually devoted to the ground states, defined as global minimizers of $E_{p,q}$ in $H_\mu^1(\R)$, i.e. functions in $H_\mu^1(\R)$ that attain
\begin{equation}
	\label{Elevel}
	\EE_{p,q}(\mu):=\inf_{u\in H_\mu^1(\R)}E_{p,q}(u)\,.
\end{equation}
Observe that, before wondering whether ground states exist, one should understand for which values of $p,q$ and $\mu$ the ground state energy level $\EE_{p,q}(\mu)$ is finite. In our setting, this is not evident due to the opposite signs of the two nonlinearities. The next theorem actually shows that the proper balance of the nonlinear terms for $\EE_{p,q}$ to be finite is determined again by the relation between $q$ and $\frac p2+1$.
\begin{theorem}
	\label{thm:propE}
	For every $p,q>2$ the ground state energy level $\EE_{p,q}:[0,+\infty)\to\R\cup\{-\infty\}$ is non--positive and non--increasing on $[0,+\infty)$. Moreover,

	\begin{enumerate}[label=(\roman*)]
		\item if $ q<\max\left\{4,\frac p2+1\right\}$, then $\EE_{p,q}(\mu)>-\infty$ for every $\mu>0$;
		
		\smallskip
		\item if $p\in(2,6)$ and $q=4$, then
		\[
		\EE_{p,4}(\mu)=\begin{cases}
		0, & \text{if }\mu\in(0,2],\\
		-\infty, & \text{if }\mu>2;
		\end{cases}
		\]
		
		\smallskip
		\item if $ q>\max\left\{4,\frac p2+1\right\}$, then $\EE_{p,q}(\mu)=-\infty$ for every $\mu>0$.
	\end{enumerate}
	Finally, if $q\neq\frac{p}{2}+1$ or $p<6$, then $\EE_{p,q}$ is continuous on every interval where it is finite
\end{theorem}
As a direct consequence of Theorem \ref{thm:propE}, ground states never exist when $ q>\max\left\{4,\frac p2+1\right\}$, i.e. regions $D$ and $E$ in Figure \ref{fig:piano_pq}, since in these regimes the energy is always unbounded from below in the mass constrained space. By \cite[Theorem 1.2]{BD21}, the same is true also when $p\in(2,6)$ and $q=4$, i.e. region $G$ in Figure \ref{fig:piano_pq}, even though $\EE_{p,q}$ is actually finite for some values of $\mu$ (see Remark \ref{rem:noGS} below for details). On the contrary, in the regimes of $p,q$ for which the ground state energy level is finite for every value of the mass, i.e. regions A, B, C, F and H in Figure \ref{fig:piano_pq}, ground states display a quite nontrivial phenomenology. Finally, note that, although Theorem \ref{thm:propE} does not manage completely region I in Figure \ref{fig:piano_pq} (only the subregion with $p\in(2,6)$), we are nevertheless able to discuss ground state existence there.

\smallskip
The next three theorems report our main results on ground states existence and multiplicity.

\begin{theorem}
	\label{thm:gs1}
	Let $p>2$ and $2< q<\min\left\{4,\frac p2+1\right\}$. The following holds:

	\begin{enumerate}[label=(\roman*)]
		
		\item if $p\in(2,6)$, then $\EE_{p,q}$ is negative on $(0,+\infty)$, strictly decreasing on $[0,\mu_{p,q})$, and $\EE_{p,q}(\mu)=\EE_{p,q}(\mu_{p,q})$ for every $\mu\geq\mu_{p,q}$, where $\mu_{p,q}$ is the value identified in Theorem \ref{thm:exsol}(i). Furthermore, ground states exist if and only if $\mu\in(0,\mu_{p,q}]$ and, for all such masses, they are unique (up to a change of sign);
			
			\smallskip
			\item if $p\geq6$, then $\EE_{p,q}$ is negative on $(0,+\infty)$ and strictly decreasing on $[0,+\infty)$, and $\displaystyle \lim_{\mu\to+\infty}\EE_{p,q}(\mu)>-\infty$. Furthermore, ground states exist and are unique (up to a change of sign) for every $\mu>0$.
	\end{enumerate}
	Moreover, there exists $\overline{\mu}_{p,q}>0$ such that $\EE_{p,q}$ is concave on $[0,\overline{\mu}_{p,q}]$ and convex on $[\overline{\mu}_{p,q},+\infty)$.
\end{theorem}

\begin{theorem}
	\label{thm:gs2}
	Let $2<p<6$ and $ \frac p2+1< q <4$, or $p>6$ and $ 4\leq q < \frac p2+1$. Then, there exists $\widetilde{\mu}_{p,q}>0$ such that
	\[
	\EE_{p,q}(\mu)\begin{cases}
	=0, & \text{if }\mu\leq\widetilde{\mu}_{p,q},\\
	<0, & \text{if }\mu>\widetilde{\mu}_{p,q},
 	\end{cases}
	\]
	$\EE_{p,q}$ is strictly decreasing on $[\widetilde{\mu}_{p,q},+\infty)$ and
	\[
	\lim_{\mu\to+\infty}\EE_{p,q}(\mu)\begin{cases}
	=-\infty, & \text{if }2<p<6,\,\frac p2+1<q<4,\\
	>-\infty, & \text{if }p>6,\,4\leq q < \frac p2+1.
	\end{cases}
	\]
	Moreover, when $q\neq 4$ ground states exist if and only if $\mu\geq\widetilde{\mu}_{p,q}$. Conversely, when $q=4$ it holds $\widetilde{\mu}_{p,4}=2$ for every $p>6$, ground states exist if and only if $\mu>2$ and, for all such masses, they are unique (up to a change of sign).
\end{theorem}

\begin{theorem}
	\label{thm:gs3}
	Let $p>2$ and $ q=\frac p2+1$. Then $\EE_{p,q}(\mu)$ is not attained for any $\mu>0$.
\end{theorem}
With respect to Figure \ref{fig:piano_pq}, Theorem \ref{thm:gs1}(i) refers to region A, Theorem \ref{thm:gs1}(ii) to B, Theorem \ref{thm:gs2} to C, F, H, and Theorem \ref{thm:gs3} to I. In view of Theorem \ref{thm:exsol} above, it is no surprise that the existence/non-existence of small mass ground states as in Theorems \ref{thm:gs1}--\ref{thm:gs2} depends only on $(p,q)$ being outside/inside the portion of the $pq$--plane enclosed by $q=4$ and $q=\frac p2+1$. Furthermore, the uniqueness of fixed mass ground states reported in Theorem \ref{thm:gs1} and in Theorem \ref{thm:gs2} with $q=4$ is a direct consequence of the uniqueness results of Theorem \ref{thm:exsol}. In these regimes, corresponding to regions A, B and H, normalized solutions and ground states coincide. Conversely, when $(p,q)$ is either in C or in F, not only we do not know whether ground states at fixed mass are unique in view of Theorem \ref{thm:nonun}, but we cannot even tell whether we have ground states for every mass for which a normalized solution exists, since we do not know whether the thresholds $\mu_{p,q}, \widetilde{\mu}_{p,q}$ found in Theorem \ref{thm:exsol}(iii) and Theorem \ref{thm:gs2} coincide (we only have the trivial estimate $\mu_{p,q}\leq\widetilde{\mu}_{p,q}$).

One of the main element of interest of Theorems \ref{thm:gs1}--\ref{thm:gs2} lies perhaps in some rather unexpected properties of the ground state energy level $\EE_{p,q}$. First, $\EE_{p,q}$ is bounded from below uniformly on $\mu>0$ whenever $q<\frac p2+1$. In particular, this automatically implies that ground states are bounded in $L^\infty(\R)$ uniformly on $\mu>0$. To the best of our knowledge, this is the first time that such phenomenon is observed for NLS equations, since usually both the $L^\infty$ norm and the energy blow up along sequences of ground states with mass diverging to infinity. Second, when one further assumes $q<4$, Theorem \ref{thm:gs1} proves that $\EE_{p,q}$ is neither concave nor convex on $(0,+\infty)$.  This feature is quite surprising, since the NLS ground state energy level has been proved to be concave not only when the nonlinearities share the same sign (see e.g. \cite{DST23}), but also with defocusing-focusing standard nonlinearities \cite[Theorem 6]{LRN20}. Thus, it seems that these new phenomena are seated not only in the competition between defocusing and focusing nonlinearities, but also in the point nature of one of them.

Observe that, when $(p,q)$ belongs to region A, the fact that $\EE_{p,q}$ is constant on $[\mu_{p,q},+\infty)$ is particularly interesting also from a variational point of view. Typically, when NLS ground states at a certain mass do not exist, this is due to the presence of an energetically convenient problem at infinity that makes energy minimizing sequences converge weakly to zero, thus losing all their mass in the limit (the vanishing case in the concentration--compactness framework).  On the contrary, this is not what happens here in region A. Indeed, since $\EE_{p,q}(\mu)=\EE_{p,q}(\mu_{p,q})$ for every $\mu>\mu_{p,q}$, but it is attained only when $\mu=\mu_{p,q}$, energy minimizing sequences with any mass $\mu$ strictly larger than the threshold will fail to be compact in $H_\mu^1(\R)$ because only the partial amount of mass $\mu-\mu_{p,q}$ runs away at infinity (with vanishing energy), while the rest of the mass concentrates towards a non-zero weak limit that is a ground state with mass $\mu_{p,q}$.

To conclude, we highlight that, even though Theorem \ref{thm:gs3} rules out ground states at every mass when $q=\frac p2+1$, no information is given about $\EE_{p,q}$ when $p\geq6$ and $ q=\frac p2+1$ in Theorem \ref{thm:propE}. On the contrary, combining Theorem \ref{thm:gs3} with Section \ref{sec:E} below, it is easy to see that $\EE_{p,q}(\mu)=0$ for every $\mu$ when $p\in(2,6)$.

\subsection*{Organization of the paper}

The paper is organized as follows:

\begin{itemize}
 \item Section \ref{sec:stat} addresses the problem \eqref{nls-nomass} without the mass constraint, providing the proof of Theorem \ref{thm:nomass};

 \smallskip
 \item Section \ref{sec:mass} investigates the dependence on $\lambda$ of the mass of the solutions of  \eqref{nls-nomass} and proves Theorems \ref{thm:exsol}--\ref{thm:nonun};

 \smallskip
 \item in Section \ref{sec:E} we derive preliminary properties of the ground state energy level, including Theorem \ref{thm:propE};

 \smallskip
 \item in Section \ref{sec:gs} we complete the proof of our main results about ground states, namely Theorems \ref{thm:gs1}--\ref{thm:gs2}--\ref{thm:gs3}.
\end{itemize}

\subsection*{Notation}

Whenever possible, in the rest of the paper we will use the shorthand notation $\|u\|_r$ for the $L^r$--norm of $u$.

\section{Proof of Theorem \ref{thm:nomass}}
\label{sec:stat}

In this section we present the proof of the first result of the paper, namely Theorem \ref{thm:nomass}. For the sake of simplicity, we split the proof in three preliminary lemmas. The first one concerns qualitative features of the solutions of \eqref{nls-nomass}.

\begin{lemma}
	\label{lem:statQual}
	Let $u\in H^1(\R)$ be a nontrivial solution of \eqref{nls-nomass}, for some $\lambda\in\R$. Then, $\lambda\geq0$. Moreover, $u$ is even and, up to a change of sign, positive on $\R$ and strictly decreasing on $\R^+$.
\end{lemma}

\begin{proof}
Let $u\in H^1(\R)$ be a nontrivial solution of \eqref{nls-nomass}, for some $\lambda\in\R$. A standard bootstrap argument shows that $u\in H^2(\R^+)\oplus H^2(\R^-)$, so that $\displaystyle \lim_{|x|\to+\infty}u(x)=\lim_{|x|\to+\infty}u'(x)=0$. Multiplying the equation in \eqref{nls-nomass} by $u'$ and suitably integrating one has
\begin{equation}
	\label{eq:emec}
	u'(x)^2=\lambda u(x)^2+\frac2p|u(x)|^p,\qquad\forall x\in\R\setminus\{0\}.
\end{equation}
Combining the continuity of $u$ on $\R$ and $u'$ on $\R\setminus\{0\}$, the boundary condition in \eqref{nls-nomass} and \eqref{eq:emec}, one sees that $u$ cannot vanish on $\R$ (otherwise, $u\equiv0$ by the Cauchy--Lipschitz Theorem) and, consequently, that  $\lambda\geq0$ (otherwise $ 0<\left(\frac{p|\lambda|}{2}\right)^{\frac{1}{p-2}}\leq |u(x)|\to0$, as $|x|\to+\infty$). Furthermore, owing again to the continuity of $u$ and $u'$, the boundary condition in \eqref{nls-nomass} and \eqref{eq:emec}, it holds that $u'(0^-)=-u'(0^+)\neq0$. Thus, with a suitable change of variable, one finds that $u$ solves somehow the same Cauchy Problem on $\R^+$ and $\R^-$, which proves even symmetry. Finally, in order to obtain decreasing monotonicity on $\R^+$ of a positive solution $u$, assume by contradiction that $u'(x)>0$ for some $x\in\R^+$. Then, \eqref{eq:emec} implies $u'(y)=\sqrt{\lambda u(y)^2+\frac{2|u(y)|^p}{p}}>0$, for every $y\geq x$, contradicting $\displaystyle \lim_{s\to+\infty}u(s)=0$.
\end{proof}

Now, in view of Lemma \ref{lem:statQual}, in order to discuss existence and multiplicity of nontrivial solutions in $H^1(\R)$ of \eqref{nls-nomass}, it suffices to focus on positive solutions. We present below such discussion dividing the cases $\lambda>0$ and $\lambda=0$ (recall that Lemma \ref{lem:statQual} also prevents the case $\lambda<0$).

\begin{lemma}
	\label{prop:l>0}
	Let $p,q>2$. The following holds:

	\begin{enumerate}[label=(\roman*)]

		\item if $ q>\frac p2+1$, then \eqref{nls-nomass} has a unique positive solution in $H^1(\R)$, for every $\lambda>0$;

		\smallskip
		\item if $ q<\frac p2+1$, then there exists $\overline\lambda_{p,q}>0$ such that \eqref{nls-nomass} has exactly:

		\begin{itemize}
		 \item two positive solutions in $H^1(\R)$ for every $\lambda\in(0,\overline{\lambda}_{p,q})$,

		 \smallskip
		 \item one positive solution in $H^1(\R)$ for $\lambda=\overline{\lambda}_{p,q}$,

		 \smallskip
		 \item no positive solution in $H^1(\R)$ for $\lambda>\overline\lambda_{p,q}$;
		\end{itemize}

		\medskip
		\item if $ q=\frac p2+1$, then, for every $\lambda>0$, \eqref{nls-nomass} has:

		\begin{itemize}
		 \item no positive solution in $H^1(\R)$ for $p\in(2,8]$,

		 \smallskip
		 \item a unique positive solution in $H^1(\R)$ for $p>8$.
		\end{itemize}
 	\end{enumerate}
\end{lemma}

\begin{proof}
	First, let $u\in H^1(\R)$ be a positive solution of \eqref{nls-nomass} for some $\lambda>0$. By Lemma \ref{lem:statQual} and \eqref{eq:emec},
	\[
	1=-\frac{u'(t)}{\sqrt{\lambda u(t)^2+\frac2p u(t)^p}},\qquad\forall t>0.
	\]
	Integrating on $(0,\overline{x})$, with $\overline{x}>0$, the change of variable $s=\big(\frac{2}{p\lambda}\big)^{\frac{1}{p-2}}u(t)$ and some simple computation yield
	\[
	 \overline{x}=-\frac{1}{\sqrt{\lambda}}\int_{\big(\frac{2}{p\lambda}\big)^{\frac{1}{p-2}}u(0)}^{\big(\frac{2}{p\lambda}\big)^{\frac{1}{p-2}}u(\overline{x})}\frac{\ds}{\sqrt{1+s^{p-2}}}.
	\]
	Thus, by the change of variable $y=s^{\frac{p-2}{2}}$ and \cite[Eq. 2.275.4, pag. 101]{GR},
	\[
	 \overline{x}=\frac{2}{\sqrt{\lambda}(p-2)}\left[\mathrm{arccoth}\left(\sqrt{\frac{2u(\overline{x})^{p-2}}{p\lambda}+1}\right)-\mathrm{arccoth}\left(\sqrt{\frac{2u(0)^{p-2}}{p\lambda}+1}\right)\right].
	\]
	Hence, suitably rearranging terms and recalling that $u$ is even, one obtains that it has to be of the form
	\begin{equation}
	\label{eq:exp_ul>0}
		u(x)=\left\{\frac{p\lambda}{2}\left[\mathrm{coth}\left(\frac{p-2}2\sqrt{\lambda}(|x|+a)\right)^2-1\right]\right\}^{\frac1{p-2}},\qquad\forall x\in\R,
	\end{equation}
	for some constant $a\in\R^+$ that satisfies
	\begin{equation}
	\label{eq:cond1}
	2\sqrt{\lambda}\,\mathrm{coth}\left(\frac{p-2}2\sqrt{\lambda}\,a\right)=u(0)^{q-2},
	\end{equation}
	where this last condition on the constant $a$ follows from the boundary condition in \eqref{nls-nomass} (and, again, the fact that $u$ is even). On the other hand, it is easy to check that a function defined as in \eqref{eq:exp_ul>0}, with $a\in\R^+$ satisfying \eqref{eq:cond1}, is a positive solution in $H^1(\R)$ of \eqref{nls-nomass}. Hence, in order to discuss existence and multiplicity of solutions of \eqref{nls-nomass} it suffices to investigate
	existence and multiplicity of the solutions $a\in\R^+$ of \eqref{eq:cond1}, on varying $\lambda\in\R^+$.

	To this aim, define the family of functions
	\begin{equation}
		\label{eq:t}
		t_\lambda(a):=\text{\normalfont coth}\left(\frac{p-2}2\sqrt{\lambda}\,a\right)
	\end{equation}
	and set $t=t_\lambda(a)$ in \eqref{eq:cond1} so that, in view of \eqref{eq:exp_ul>0}, one obtains
	\begin{equation}
		\label{eq:condt}
		\frac{t}{(t^2-1)^{\frac{q-2}{p-2}}}=\frac{1}{2}\left(\frac p2\right)^{\frac{q-2}{p-2}}\lambda^{\frac{2q-p-2}{2(p-2)}}\,.
	\end{equation}
	Note that, since $t_\lambda(\cdot)$ is a continuous and strictly decreasing bijection of $\R^+$ onto $(1,+\infty)$, for every $\lambda>0$, we can equivalently study existence and multiplicity of the solutions $t\in(1,+\infty)$ of \eqref{eq:condt}, on varying $\lambda\in\R^+$.
	Now, let
	\begin{equation}
	\label{eq:f e g}
	f(t):=\frac{t}{(t^2-1)^{\frac{q-2}{p-2}}}\qquad\text{and}\qquad g(\lambda):=\frac{1}{2}\left(\frac p2\right)^{\frac{q-2}{p-2}}\lambda^{\frac{2q-p-2}{2(p-2)}}\,,
	\end{equation}
	so that \eqref{eq:condt} becomes $f(t)=g(\lambda)$. Clearly,
	\[
	f'(t)=\frac{\frac{p+2-2q}{p-2}t^2-1}{(t^2-1)^{\frac{p+q-4}{p-2}}}\qquad\text{and}\qquad g'(\lambda)=\frac{2q-p-2}{4(p-2)}\left(\frac p2\right)^{\frac{q-2}{p-2}}\lambda^{\frac{2q-3p+2}{2(p-2)}}.
	\]

	If $\displaystyle q>\frac p2+1$, then $\displaystyle \lim_{t\to 1^+}f(t)=+\infty$, $\displaystyle \lim_{t\to+\infty}f(t)=0$, and $f'(t)<0$ for every $t>1$, whereas $\displaystyle \lim_{\lambda\to0^+}g(\lambda)=0$, $\displaystyle \lim_{\lambda\to+\infty}g(\lambda)=+\infty$, and $g'(\lambda)>0$ for every $\lambda>0$. Hence, for every fixed $\lambda>0$, \eqref{eq:condt} is satisfied for a unique value of $t>1$, which proves point (i).

	If, on the contrary, $ q<\frac p2+1$, then $\displaystyle \lim_{t\to1^+}f(t)=\lim_{t\to+\infty}f(t)=+\infty$ and $f'$ has a unique zero in $(1,+\infty)$, whereas $\displaystyle \lim_{\lambda\to0^+}g(\lambda)=+\infty$, $\displaystyle \lim_{\lambda\to+\infty}g(\lambda)=0$, and $g'(\lambda)<0$ for every $\lambda>0$. Hence, denoting by $\overline\lambda_{p,q}>0$ the unique solution of the equation $\displaystyle g(\lambda)=\min_{t\in(1,+\infty)}f(t)$, we obtain that \eqref{eq:condt} is satisfied for exactly two values of $t$ in $(1,+\infty)$ when $\lambda\in(0,\overline\lambda_{p,q})$, for exactly one value of $t$ in $(1,+\infty)$ when $\lambda=\overline\lambda_{p,q}$, and for no value of $t\in(1,+\infty)$ when $\lambda>\overline\lambda_{p,q}$, which proves point (ii).

	Finally, if $ q=\frac p2+1$, then
	\begin{equation}
	\label{eq:fgdiag}
	f(t)=\frac{t}{\sqrt{t^2-1}}\qquad\text{and}\qquad g(\lambda)=\frac{\sqrt{p}}{2\sqrt{2}}.
	\end{equation}
	In particular, $\displaystyle \lim_{t\to1^+}f(t)=+\infty$, $\displaystyle \lim_{t\to+\infty}f(t)=1$, and $f'(t)<0$ for every $t>1$, whereas $g(\lambda)$ is constant on $\R^+$. Hence, \eqref{eq:condt} is satisfied by a value of $t$ if and only if the quantity $\frac{\sqrt{p}}{2\sqrt{2}}$ is greater than 1, and in this case such value is unique. Since this requires $p>8$, also the proof of point (iii) is complete.
\end{proof}

\begin{lemma}
	\label{prop:l=0}
	Let $p,q>2$ and $\lambda=0$. Then, \eqref{nls-nomass} admits a positive solution in $H^1(\R)$ if and only if $p\in(2,6)$ and $ q\neq\frac p2+1$, and such solution is unique.
\end{lemma}

\begin{proof}
	Arguing as in the proof of Lemma \ref{prop:l>0}, one can check that $u\in H^1(\R)$ is a positive solution of \eqref{nls-nomass} with $\lambda=0$ if and only if $p<6$ and $u$ is of the form
	\begin{equation}
		\label{eq:expul=0}
		u(x)=\frac{c_p}{(|x|+a)^{\frac2{p-2}}},\qquad\forall x\in\R,\qquad\text{with}\quad c_p:=\left(\frac{\sqrt{2p}}{p-2}\right)^{\frac{2}{(p-2)}},
	\end{equation}
	for some constant $a\in\R^+$ such that
	\begin{equation}
		\label{eq:condl=0}
	c_p^{q-2}=\frac{4a^{\frac{2q-p-2}{p-2}}}{p-2},
	\end{equation}
	(note that, for $p\geq6$, positive solutions are still of the form \eqref{eq:expul=0}, but do not belong to $L^2(\R)$). Hence, existence and multiplicity of positive solutions of \eqref{nls-nomass} can be reduced again to existence and multiplicity of solutions $a\in\R^+$ to \eqref{eq:condl=0}. If $ q\neq\frac p2+1$, then it is immediate to see that \eqref{eq:condl=0} has a unique positive solution. Conversely, if $ q=\frac p2+1$, then \eqref{eq:condl=0} has no solution whenever $p\in(2,6)$ (it is always fulfilled when $p=8$, but this give rise to a solution that is not in $H^1(\R)$).
\end{proof}

Finally, we can sum up all the previous results to obtain

\begin{proof}[Proof of Theorem \ref{thm:nomass}]
 The first features of the solutions and point (i) follow by Lemma \ref{lem:statQual}; whereas, points (ii) and (iii) follow by Lemmas \ref{prop:l=0} and \ref{prop:l>0}, respectively.
\end{proof}

\section{Proofs of Theorems \ref{thm:exsol}--\ref{thm:nonun}}
\label{sec:mass}

In this section we present the proofs of Theorems \ref{thm:exsol}--\ref{thm:nonun}. Note that throughout the section we focus only on positive and even solutions of \eqref{nlse}, strictly decreasing on $\R^+$, since this is not restrictive by Lemma \ref{lem:statQual}.

The proofs of the mentioned results relies on a detailed analysis of the masses of the stationary states identified in the previous section. In the case $\lambda=0$, letting $u_0$ be the sole positive solution in $H^1(\R)$ of \eqref{nls-nomass} provided by Lemma \ref{prop:l=0}, a straightforward computation shows
\begin{equation}
 \label{eq-mu0}
 \|u_0\|_2^2=\frac{2^{\frac{3(q-p+2)}{2q-p-2}}p^{\frac{q-4}{2q-p-2}}}{6-p}=:\mu_0.
\end{equation}
In the case $\lambda>0$ the computation is far from being immediate and, to obtain some information, we essentially use the fact that, by the argument of the proof of Lemma \ref{prop:l>0}, given $p,q>2$, each solution $u$ of \eqref{nls-nomass}, with $\lambda>0$ fixed, corresponds to a value $t\in(1,+\infty)$ that satisfies \eqref{eq:condt}. In particular, we exploit this relation to obtain an explicit formula for the mass of $u$ in terms of the associated value of $t$.

Preliminarily, recall that, by \eqref{eq:exp_ul>0}, a solution $u$ of \eqref{nls-nomass} with $\lambda>0$ satisfies
\[
\|u\|_{2}^2=2\int_{\R^+}|u(x)|^2\,dx= 2\left(\frac{p\lambda}2\right)^{\frac2{p-2}}\int_a^{+\infty}\left(\mathrm{coth}\left(\frac{p-2}{2}\sqrt{\lambda}\,x\right)^2-1\right)^{\frac2{p-2}}\,dx\,,
\]
By the change of variable $\displaystyle s=\mathrm{coth}\left(\frac{p-2}2\sqrt{\lambda}\,x\right)$, one can check that the previous formula reads
\begin{equation}
\label{eq:mass1}
\|u\|_{2}^2=\frac{2^{\frac{2p-6}{p-2}}p^{\frac{2}{p-2}}}{p-2}\lambda^{\frac{6-p}{2(p-2)}}\int_1^t(s^2-1)^{\frac{4-p}{p-2}}\,ds\,,
\end{equation}
where $t$ is the number in $(1,+\infty)$ associated to $u$ by \eqref{eq:condt} (in view of \eqref{eq:t}), which depends in general by $\lambda$. However, it is worth recalling that, whenever $p>8$ and $ q=\frac p2+1$, the proof of Lemma \ref{prop:l>0} shows that $t$ depends in fact only on $p$ and $q$, and is independent of $\lambda$. This means that, in this regime of $p,q$, the mass of the solution $u$ of \eqref{nls-nomass}, with $\lambda>0$ fixed, is given by
\begin{equation}
	\label{eq:mu_diag}
	\|u\|_{2}^2=M_{p,q}\,\lambda^{\frac{6-p}{2(p-2)}}=:\mu(\lambda),
\end{equation}
with
\[
M_{p,q}:=\frac{2^{\frac{2p-6}{p-2}}p^{\frac{2}{p-2}}}{p-2}\int_1^t(s^2-1)^{\frac{4-p}{p-2}}\,ds\,.
\]
Thus, we can start characterizing the mass of the stationary states by this last, and easy, case.

\begin{lemma}
	\label{lem:mu_diag}
	Let $p>8$ and $ q=\frac p2+1$. Then, for every $\mu>0$, there exists a unique positive solution of \eqref{nls-nomass} in $H_\mu^1(\R)$ ($H_\mu^1(\R)$ being defined by \eqref{eq-vincolo}).
\end{lemma}
\begin{proof}
	It is sufficient to note that the map $\lambda\mapsto\mu(\lambda)$ defined in \eqref{eq:mu_diag} is a bijection of $\R^+$ onto itself.
\end{proof}

Let us, now, focus on the case $ q\neq\frac p2+1$. Using the one--to--one correspondence between solutions $u$ of \eqref{nls-nomass}, with $\lambda>0$ fixed, and solutions $t\in(1,+\infty)$ of \eqref{eq:condt}, the relation between $t$ and $\lambda$ given again by \eqref{eq:condt} and the definition of $f(t)$ given by \eqref{eq:f e g}, we can rewrite the right hand side of \eqref{eq:mass1} as a function of $t$ (note that this is possible since the function $g$ defined by \eqref{eq:f e g} is a bijection of $\R^+$ onto itself whenever $ q\neq\frac p2+1$). Precisely,
\[
 \|u\|_2^2=\mu(t)
\]
with $\mu:(1,+\infty)\to\R^+$ given by
\begin{equation}
\label{eq:mut}
\mu(t):=C_{p,q}f(t)^{\frac{6-p}{2q-p-2}}I(t),
\end{equation}
where
\begin{equation}
\label{eq:It}
I(t):=\int_1^t (s^2-1)^{\frac{4-p}{p-2}}\,ds\qquad\text{and}\qquad C_{p,q}:=\frac{2^{\frac{3(q-p+2)}{2q-p-2}}p^{\frac{q-4}{2q-p-2}}}{p-2}.
\end{equation}
Note also that $I(t)$ is well defined since $\frac{4-p}{p-2}>-1$ for every $p>2$, and that, in view of \eqref{eq-mu0}, if $p\in(2,6)$ and $q>2$, (with $\displaystyle q\neq\frac{p}{2}+1$), then
\begin{equation}
\label{eq-mucpq}
 \mu_0=\frac{p-2}{6-p}\,C_{p,q}.
\end{equation}

It is, thus, crucial to establish the qualitative behavior of $\mu(t)$ in $(1,+\infty)$. We begin with the analysis of its asymptotic behaviours.

\begin{lemma}
	\label{lem:mu_as}
	Let $p,q>2$ with $ q\neq\frac p2+1$. It holds that
	\begin{equation}
	\label{eq-as-t1}
	\mu(t)\sim p^{\frac{q-4}{2q-p-2}}2^{\frac{6-p}{2q-p-2}}(t-1)^{\frac{q-4}{2q-p-2}},\qquad\text{as}\quad t\to1^+.
	\end{equation}
	Moreover, there exist $c_{p,q}>0$ (depending on $p,q$) such that
	\begin{equation}
	\label{eq-as-tinf}
	\mu(t)\sim\begin{cases}
		\mu_0, & \text{if }p\in(2,6),\\
		C_{6,q}\log(t), & \text{if }p=6,\\
		c_{p,q}\,t^{\frac{p-6}{p-2}}, & \text{if }p>6,
	\end{cases}
	\qquad\text{as}\quad t\to+\infty,
	\end{equation}
	with $\mu_0$ defined by \eqref{eq-mu0} and $C_{6,q}$ defined by \eqref{eq:It}.
\end{lemma}

\begin{proof}
	By \eqref{eq:mut}, the analysis of the asymptotic behaviour of $\mu(t)$ relies on those of $f(t)$ (defined by \eqref{eq:f e g}) and $I(t)$ (defined by \eqref{eq:It}). As for $f(t)$, it is easy to check that
	\begin{equation}
	\label{eq:as_f}
	f(t)\sim\begin{cases}
		2^{-\frac{q-2}{p-2}}(t-1)^{-\frac{q-2}{p-2}}, & \quad\text{as}\quad t\to1^+,\\
		t^{\frac{p+2-2q}{p-2}}, & \quad\text{as}\quad t\to+\infty.
	\end{cases}
	\end{equation}
	Let us turn to $I(t)$. A straightforward application of de l'H$\hat{\text{o}}$pital shows that
	\begin{equation}
	\label{eq:It1+}
	I(t)\sim 2^{\frac{6-2p}{p-2}}(p-2)(t-1)^{\frac2{p-2}},\qquad\text{as}\quad t\to1^+,
	\end{equation}
	which, together with \eqref{eq:as_f}, yields \eqref{eq-as-t1}.

	Focus, then, on the case $t\to+\infty$. First, discuss powers $p\in(2,6)$. Note that, if $p=4$, then $I(t)=t-1$, that combined with \eqref{eq:as_f} gives $\displaystyle f(t)^{\frac{6-p}{2q-p-2}}I(t)\to1$, as $t\to+\infty$, whence we conclude in view of \eqref{eq-mucpq}. Conversely, if $p\in(2,4)$, then we have
	\[
	\int_1^t(s-1)^{\frac{2(4-p)}{p-2}}\,ds \leq I(t) \leq \int_1^t(s+1)^{\frac{2(4-p)}{p-2}}\,ds,
	\]
	whereas, if $p\in(4,6)$, then it holds
	\begin{equation}
	\label{eq-stimaI2}
	\int_1^t(s+1)^{\frac{2(4-p)}{p-2}}\,ds \leq I(t) \leq \int_1^t(s-1)^{\frac{2(4-p)}{p-2}}\,ds .
	\end{equation}
	Computing explicitly the integrals in the previous estimates and coupling again with \eqref{eq:as_f} entail, for every $p\in(2,6)$,
	\[
	f(t)^{\frac{6-p}{2q-p-2}}I(t)\sim \frac{p-2}{6-p}\qquad\text{as}\quad t\to+\infty,
	\]
	whence, again, we conclude by \eqref{eq-mucpq}.

	Let, then, $p\geq 6$ and write
	\[
	I(t)=\int_1^2(s^2-1)^{\frac{4-p}{p-2}}\,ds+\int_2^t(s^2-1)^{\frac{4-p}{p-2}}\,ds.
	\]
	As the former integral is finite for every $p\geq6$, it is sufficient to focus on the latter. Here, using (again) de l'H$\hat{\text{o}}$pital for $p=6$ and the fact that $\displaystyle \frac{4-p}{p-2}<-\frac12$ for $p>6$, we find that
	\[
	\int_2^t(s^2-1)^{\frac{4-p}{p-2}}\,ds\sim\begin{cases}
		\log(t), & \text{if }p=6,\\
		K_p, & \text{if }p>6,
	\end{cases}
	\qquad\text{as}\quad t\to+\infty,
	\]
	so that
	\[
	I(t)\sim\begin{cases}
		\log(t) & \text{if }p=6\\
		K_p, & \text{if }p>6,
	\end{cases}
	\qquad\text{as}\quad t\to+\infty,
	\]
	for a suitable constant $K_{p}>0$. Combining again with \eqref{eq:as_f}, the proof is complete.
\end{proof}

In the next lemma we establish the regimes of $p,q$ where $\mu(t)$ is strictly increasing.

\begin{lemma}
	\label{lem:mon_mu}
	Let $p>2$. If
	\begin{equation}
	\label{eq-assincreasing}
	q\in(2,4]\text{ and }q<\frac p2+1\,,\qquad\text{or}\qquad q\geq4\text{ and }q>\frac p2+1\,,
	\end{equation}
	the map $t\mapsto\mu(t)$ defined in \eqref{eq:mut} is strictly increasing on $(1,+\infty)$.
\end{lemma}
\begin{proof}
	By definition, $\mu(t)$ is of class $C^1$ on $(1,+\infty)$ and
	\[
	\mu'(t)=C_{p,q}\frac{f(t)^{\frac{2(4-q)}{2q-p-2}}}{(t^2-1)^{\frac{p+q-6}{p-2}}}h(t)\,,
	\]
	where $C_{p,q}$ is as in \eqref{eq:It} and
	\begin{equation}
	\label{eq:h}
	h(t):=\frac{6-p}{p-2}\frac{\frac{p-2}{p+2-2q}-t^2}{(t^2-1)^{\frac2{p-2}}}I(t)+t\,.
	\end{equation}
	Since $f(t)$ and $t^2-1$ are strictly positive for every $t\in(1,+\infty)$, in order to conclude it suffices to establish that, if $p,\,q$ satisfy \eqref{eq-assincreasing}, then $h(t)>0$ for every $t>1$. Note preliminarily this is straightforward when $p\geq 6$ and $ q>\frac p2+1$. We divide the discussion on the remaining range of powers in three parts.
	
	\smallskip
	\emph{Part 1: $p\in(2,6)$ and $ 2<q<\frac p2+1$}. Since $ t^2\leq\frac{p-2}{p+2-2q}$ immediately yields $h(t)>0$, it is left to study the case $ t^2>\frac{p-2}{p+2-2q}$. Assume first $p<4$. Integrating by parts we have
	\[
	I(t)= t(t^2-1)^{\frac{4-p}{p-2}}-\frac{2(4-p)}{p-2}\int_1^t s^2(s^2-1)^{\frac{6-2p}{p-2}}\,ds.
	\]
	Since, for every $s>1$,
	\[
	s^2(s^2-1)^{\frac{6-2p}{p-2}}=(s^2-1)^{\frac{4-p}{p-2}}+(s^2-1)^{\frac{6-2p}{p-2}},
	\]
	plugged into the previous identity we obtain
	\[
	I(t)=\frac{p-2}{6-p}t(t^2-1)^{\frac{4-p}{p-2}}-\frac{2(4-p)}{6-p}\int_1^t(s^2-1)^{\frac{6-2p}{p-2}}\,ds.
	\]
	Thus, combined with \eqref{eq:h}, we find
	\[
	h(t)=\frac{2(q-2)t}{(p+2-2q)(t^2-1)}-\frac{2(4-p)}{p-2}\frac{\frac{p-2}{p+2-2q}-t^2}{(t^2-1)^{\frac2{p-2}}}\int_1^t(s^2-1)^{\frac{6-2p}{p-2}}\,ds,
	\]
	which shows that, $h(t)>0$. Assume, then, $p\in[4,6)$. By the second inequality in \eqref{eq-stimaI2} we have $I(t)\leq\frac{p-2}{6-p}(t-1)^{\frac{6-p}{p-2}}$ for every $t>1$, so that
	\begin{multline*}
	h(t)\geq\frac{\frac{p-2}{p+2-2q}-t^2}{(t^2-1)^{\frac2{p-2}}}(t-1)^{\frac{6-p}{p-2}}+t>\frac{1-t^2}{(t^2-1)^{\frac2{p-2}}}(t-1)^{\frac{6-p}{p-2}}+t\\
	=t-(t^2-1)^{\frac{p-4}{p-2}}(t-1)^{\frac{6-p}{p-2}}=t-(t+1)^{\frac{p-4}{p-2}}(t-1)^{\frac{2}{p-2}},
	\end{multline*}
	the second inequality coming from the fact that $ \frac{p-2}{p+2-2q}>1$. Since $\frac{p-4}{p-2}<\frac12$ and $\frac{2}{p-2}<\frac12$, the right--hand side of the previous estimate is strictly positive for every $t>1$ and so, also in this case, $h(t)>0$, which concludes the proof of Part 1.
	
	\smallskip
	\textit{Part 2: $p\geq 6$ and $ 2<q\leq4$.}  When $p=6$ there is nothing to prove, since $h(t)=t>0$ on $(1,+\infty)$. Assume, then, $p>6$. In this case, it is obvious that $h(t)>0$, whenever $ t^2\geq \frac{p-2}{p+2-2q}$. Let us focus, then, on the values of $t>1$ such that $ t^2<\frac{p-2}{p+2-2q}$. Since
	\[
	I(t)<2^{\frac{4-p}{p-2}}\int_1^t(s-1)^{\frac{4-p}{p-2}}\ds=2^{\frac{6-2p}{p-2}}(p-2)(t-1)^{\frac{2}{p-2}},
	\]
	reminding that $ 1<t< \sqrt{\frac{p-2}{p+2-2q}}$, we obtain
	\[
	h(t)>2^{\frac{6-2p}{p-2}}(6-p)\frac{\frac{p-2}{p+2-2q}-t^2}{(t+1)^{\frac2{p-2}}}+t> 2^{\frac{6-2p}{p-2}}(6-p)\frac{\frac{p-2}{p+2-2q}-1}{2^{\frac2{p-2}}}+1=\frac{(p-2)(4-q)}{2(p+2-2q)}\geq0,
	\]
	where we used that the function $ s\mapsto 2^{\frac{6-2p}{p-2}}(6-p)\frac{\frac{p-2}{p+2-2q}-s^2}{(s+1)^{\frac2{p-2}}}+s$ is strictly increasing on the interval $ \left(1,\sqrt{\frac{p-2}{p+2-2q}}\right)$.
	
	\smallskip
	\textit{Part 3: $p\in(2,6)$ and $q\geq4$.} By \eqref{eq:h} and the assumptions on $p,\,q$, in order to prove that $h(t)>0$ for every $t>1$, it is enough to check that
	\begin{equation}
	\label{eq:phi12}
	\varphi_1(t):=\frac{I(t)}{(t^2-1)^{\frac2{p-2}}}<\frac{p-2}{6-p}\frac{t}{t^2+\frac{p-2}{2q-p-2}}=:\varphi_2(t),\qquad\forall t>1.
	\end{equation}
	Observe that $ \varphi_2(1)=\frac{(p-2)(2q-p-2)}{2(6-p)(q-2)}$, whereas \eqref{eq:It1+} gives $\displaystyle \lim_{t\to1^+}\varphi_1(t)=\tfrac{p-2}{4}$, so that it is easy to see that $\displaystyle\varphi_2(1)\geq\lim_{t\to1^+}\varphi_1(t)$ whenever $p\in(2,6)$ and $q\geq4$. Moreover,
	\[
	\varphi_1'(t)=\frac{1}{t^2-1}\left(1-\frac{4tI(t)}{(p-2)(t^2-1)^{\frac2{p-2}}}\right)\qquad\text{and}\qquad\varphi_2'(t)=\frac{p-2}{6-p}\frac{\frac{p-2}{2q-p-2}-t^2}{\left(t^2+\frac{p-2}{2q-p-2}\right)^2}\,.
	\]
	On the one hand, this immediately implies that $\varphi_2$ is strictly increasing on $\left(1,\sqrt{\frac{(p-2)}{(2q-p-2)}}\right)$ and strictly decreasing on $\left(\sqrt{\frac{(p-2)}{(2q-p-2)}},+\infty\right)$. On the other hand, note that $I(t)\sim\frac{p-2}{4}\frac{(t^2-1)^{\frac2{p-2}}}{t}$, as $t\to1^+$, and
	\[
	I'(t)=(t^2-1)^{\frac{4-p}{p-2}}> \frac{(t^2-1)^{\frac{4-p}{p-2}}}{t}-\frac{p-2}{4}\frac{(t^2-1)^{\frac2{p-2}}}{t^2}=\left(\frac{p-2}{4}\frac{(t^2-1)^{\frac2{p-2}}}{t}\right)'
	\]
	for every $t>1$. Hence, $\varphi_1'(t)<0$ for every $t>1$, so that $\varphi_1$ is strictly decreasing on $(1,+\infty)$.
	
	Now, assume by contradiction that \eqref{eq:phi12} is not satisfied. By the monotonicity of $\varphi_1$ and $\varphi_2$, this implies the existence of $\overline{t}>1$ such that
	\begin{gather}
	 \label{eq-phieq}\varphi_1(\overline{t})=\varphi_2(\overline{t}).\\
	 \label{eq-phiineq} \varphi_1'(\overline{t})\geq\varphi_2'(\overline{t}).
	\end{gather}
    Using \eqref{eq:phi12} and \eqref{eq-phieq}, $I(\overline{t})=\varphi_2(\overline{t})(\overline{t}^2-1)^{\frac{2}{p-2}}$, so that \eqref{eq-phiineq} reads
	\[
	\begin{split}
		 \frac{1}{\overline{t}^2-1}\left(1-\frac{4\overline{t}^2}{(6-p)\left(\overline{t}^2+\frac{p-2}{2q-p-2}\right)}\right)\geq\frac{p-2}{6-p}\frac{\frac{p-2}{2q-p-2}-\overline{t}^2}{\left(\overline{t}^2+\frac{p-2}{2q-p-2}\right)^2}.
	\end{split}
	\]
	Thus, expanding computations and multiplying both sides times $-\frac{(6-p)[(2q-p-2)\overline{t}^2+p-2]}{p-2}$, which is negative for the chosen range of $p,\,q$, we have
	\[
	 \frac{(2q-p-2)\overline{t}^2-(6-p)}{\overline{t}^2-1}\leq(2q-p-2)-\frac{2(2q-p-2)(p-2)}{(2q-p-2)\overline{t}^2+(p-2)}.
	\]
	Hence, multiplying times $\overline{t}^2-1$ and suitably rearranging terms, we obtain
	\[
	 \frac{(2q-p-2)(p-2)(\overline{t}^2-1)}{(2q-p-2)\overline{t}^2+(p-2)}\leq 4-q,
	\]
	which is a contradiction as the left hand side positive and the right hand side is nonpositive for the chosen range of $p,\,q$.
\end{proof}

\begin{remark}
	\label{rem:nomon_mu}
	Note that Lemma \ref{lem:mon_mu} does not cover the cases $ p\in(2,6), \frac p2+1 < q < 4$, and $ p>6, 4<q<\frac p2+1$. This is not due to a flaw in the result, but it is rather seated in the structure of the problem. Indeed, in these regimes the map $t\mapsto\mu(t)$ defined in \eqref{eq:mut} is not monotone in general. In the latter case, this is an immediate consequence of Lemma \ref{lem:mu_as}, which entails $\displaystyle \lim_{t\to1^+}\mu(t)=\lim_{t\to+\infty}\mu(t)=+\infty$. In the former case, the study of the monotonicity for general $p,q$ can be rather hard. However, taking e.g. $p=4$ yields $I(t)=t-1$, so that the function $h(t)$ defined by \eqref{eq:h} reads $h(t)=\frac{(q-3)t-1}{(q-3)(t+1)}$, that changes sign on $(1,+\infty)$ for every $q\in(3,4)$.
\end{remark}

Now, we have all the elements to prove Theorems \ref{thm:exsol}--\ref{thm:nonun}.

\begin{proof}[Proof of Theorem \ref{thm:exsol}]
	If $2<p<6$ and $ 2<q<\frac p2+1$, then, by Theorem \ref{thm:nomass}(ii) and Theorem \ref{thm:nomass}(ii)(iii)(c), \eqref{nls-nomass} admits a nontrivial solution in $H^1(\R)$ if and only if $\lambda\in[0,\overline{\lambda}_{p,q}]$. Moreover, by Lemma \ref{lem:mu_as} the mass of such solutions satisfies
	\[
	\lim_{t\to1^+}\mu(t)=0\qquad\text{and}\qquad\lim_{t\to+\infty}\mu(t)=\mu_0\,,
	\]
	where $\mu_0$ is the mass of the solution of \eqref{nls-nomass} with $\lambda=0$ defined by \eqref{eq-mu0}. Since $\mu(t)$ is continuous and, by Lemma \ref{lem:mon_mu}, strictly increasing on $(1,+\infty)$, this shows that \eqref{nlse} admits a nontrivial solution if and only if $\mu\in(0,\mu_{p,q}]$, with $\mu_{p,q}=\mu_0$, and that such solution is unique. Conversely, if $2<p<6$ and $q>4$, then, by Theorem \ref{thm:nomass}(iii)(a), \eqref{nls-nomass} admits a nontrivial solution in $H^1(\R)$ for every $\lambda\geq0$. However, since Lemmas \ref{lem:mu_as} and \ref{lem:mon_mu} provide the same asymptotics and monotonicity for $\mu(t)$, respectively, one obtains the same result, which thus completes the proof of Theorem \ref{thm:exsol}(i) together with the associated uniqueness claim.
	
	If $p\geq6$, then, by Theorem \ref{thm:nomass}(ii) and Theorem \ref{thm:nomass}(iii), we have that \eqref{nls-nomass} admits a nontrivial solution in $H^1(\R)$ only for $\lambda\in(0,\overline{\lambda}_{p,q}]$, when $q\in(2,4)$, and only for $\lambda>0$, when $ q>\frac p2+1$. Since in these cases Lemma \ref{lem:mu_as} yields
	\[
	\lim_{t\to1^+}\mu(t)=0\qquad\text{and}\qquad\lim_{t\to+\infty}\mu(t)=+\infty,
	\]
	again by the continuity of $\mu(t)$ one obtains Theorem \ref{thm:exsol}(ii). In addition, the associated uniqueness claim follows again by the monotonicity established by Lemma \ref{lem:mon_mu}.
	
	Let, now, $p\in(2,6)$ and $ \frac p2+1 < q < 4$. By Theorem \ref{thm:nomass}, \eqref{nls-nomass} admits a nontrivial solution in $H^1(\R)$ for every $\lambda\geq0$. Also, Lemma \ref{lem:mu_as} shows that
	\[
	\lim_{t\to1^+}\mu(t)=+\infty\qquad\text{and}\qquad\lim_{t\to+\infty}\mu(t)=\mu_0.
	\]
	By continuity, this shows that $\mu(t)$ is uniformly bounded away from zero on $(1,+\infty)$. Hence, \eqref{nlse} admits a solution if and if $\mu\geq\mu_{p,q}$, for some $\mu_{p,q}>0$. On the other hand, when $p>6$ and $ 4< q <\frac p2+1$, by Theorem \ref{thm:nomass}(ii) and Theorem \ref{thm:nomass}(iii)(a), \eqref{nls-nomass} admits nontrivial solution only for $\lambda\in[0,\overline{\lambda}_{p,q}]$. However, since the behavior of $\mu(t)$ is completely analogous, with the only difference that $\displaystyle \lim_{t\to+\infty}\mu(t)=+\infty$, Theorem \ref{thm:exsol}(iii) follows.
	
	When $p\in(2,6)$ and $q=4$, \eqref{nls-nomass} admits a nontrivial solution in $H^1(\R)$ for every $\lambda\geq0$ again by Theorem \ref{thm:nomass}(ii) and Theorem \ref{thm:nomass}(iii)(a). Moreover, Lemma \ref{lem:mu_as} gives
	\[
	\lim_{t\to1^+}\mu(t)=2\qquad\text{and}\qquad\lim_{t\to+\infty}\mu(t)=\mu_0\,,
	\]
	and Lemma \ref{lem:mon_mu} guarantees that $\mu(t)$ is strictly increasing on $(1,+\infty)$. Hence, \eqref{nlse} admits a solution for every $\mu\in(2,\mu_0]$, which proves Theorem \ref{thm:exsol}(iv), and such solution is unique. Theorem \ref{thm:exsol}(v) and the associated uniqueness claim can be proved in the very same way, simply replacing $\mu_0$ with $+\infty$.

	Finally, since the content of Theorem \ref{thm:exsol}(vi) is proved by Lemma \ref{lem:mu_diag} and Theorem \ref{thm:nomass}(iii), we conclude.
\end{proof}

\begin{proof}[Proof of Theorem \ref{thm:nonun}]
	It is an immediate consequence of Remark \ref{rem:nomon_mu}.
\end{proof}	

\begin{remark}
	\label{rem:cont_sol}
	Since we need it later, let us focus on the map that associates each $t\in (1,+\infty)$ with the unique positive solution $u_t\in H^1(\R)$ of \eqref{nls-nomass} with $\lambda>0$ obtained through \eqref{eq:exp_ul>0}--\eqref{eq:condt}, and $+\infty$ with the solution $u_{\infty}\in H^1(\R)$ of \eqref{nls-nomass} with $\lambda=0$ given by \eqref{eq:expul=0} (note that this last abuse of notation is consistent since, setting $t=+\infty$ in \eqref{eq:condt}, one obtains $\lambda=0$). One can check that this map is continuous from any interval $I\subseteq(1,+\infty]$ on which it is defined to $H^1(\R)$. Indeed, by \eqref{eq:exp_ul>0}--\eqref{eq:condt} and \eqref{eq:expul=0}, one can see that $u_t\to u_{\overline{t}}$ pointwise on $\R$ as $t\to\overline{t}\in I$. Moreover, since $(u_t)_{t\in I}$ is equibounded in $H^1(\R)$ and every $u_t$ is radially decreasing on $\R$ (by Lemma \ref{lem:statQual}), we have $u_t\to u_{\overline{t}}$ in $L^r(\R)$ for every $r\in(2,+\infty]$, whereas the convergence of $u_t$ to $u_{\overline{t}}$ in $L^2(\R)$ is a direct consequence of the continuity of the map $\mu(t)$ defined by \eqref{eq:mut}. Given that \eqref{eq:condt} also ensures that each value of $t$ is associated with a single $\lambda(t)$ and that the map $t\mapsto \lambda(t)$ is continuous, passing to the limit as $t\to\overline{t}$ in \eqref{nls-nomass} and exploiting the convergences proved so far one can prove that $u_t\to u_{\overline{t}}$ in $H^1(\R)$.
\end{remark}

\section{Proof of Theorem \ref{thm:propE}}
\label{sec:E}

In this section we establish some preliminary features of the ground state energy level $\EE_{p,q}$, defined by \eqref{Elevel}, which play an important role in the rest of the paper. Precisely, we prove Theorem \ref{thm:propE}.

Since we frequently use them throughout, we recall here the following Gagliardo--Nirenberg inequalities:
\begin{equation}
	\label{eq:GNp}
	\|u\|_p^p\leq K_p\|u\|_2^{\frac{p}{2}+1}\|u'\|_2^{\frac p2-1},\qquad\forall u\in H^1(\R),\qquad p>2,
\end{equation}
for suitable constants $K_p>0$, and
\begin{equation}
	\label{eq:Gninf}
	\|u\|_\infty^2\leq\|u\|_2\|u'\|_2,\qquad\forall u\in H^1(\R).
\end{equation}
Furthermore, observe that for every $p,q>2$ and $\mu_1,\mu_2>0$, if one takes $u\in H_{\mu_1}^1(\R)$ and sets
\[
v(x):=\left(\tfrac{\mu_2}{\mu_1}\right)^\alpha u\left(x\left(\tfrac{\mu_2}{\mu_1}\right)^{2\alpha-1}\right),\qquad\forall x\in\R,
\]
for any given $\alpha>0$, then one obtains that $v\in H_{\mu_2}^1(\R)$ and
\begin{equation}
\label{eq:scalEst}
\EE_{p,q}(\mu_2)\leq E_{p,q}(v)=\frac12\left(\frac{\mu_2}{\mu_1}\right)^{4\alpha-1}\|u'\|_{2}^2+\frac1p\left(\frac{\mu_2}{\mu_1}\right)^{\alpha(p-2)+1}\|u\|_{p}^p-\frac1q\left(\frac{\mu_2}{\mu_1}\right)^{\alpha q}|u(0)|^q.
\end{equation}

\begin{remark}
	\label{rem:E<0_dec}
	Observe that, for every $p,q>2$, the map $\mu\mapsto\EE_{p,q}(\mu)$ is non--positive and non--increasing on $(0,+\infty)$. To see that $\EE_{p,q}(\mu)\leq0$ for every $p,q>2$ and $\mu>0$ one can argue as follows.  Let $(u_n)_n\subset H_\mu^1(\R)$ be such that $u_n\equiv c_n$ on $[-n^2,n^2]$, $u_n\equiv0$ on $\R\setminus[-n^2-1,n^2+1]$, and $u_n$ is linear both on $[n^2,n^2+1]$ and on $[-n^2-1,-n^2]$. Here, $c_n>0$ is chosen to guarantee that $\|u_n\|_{2}^2=\mu$. Since a direct computation easily shows that $c_n\sim n^{-1}\sqrt{\frac{\mu}{2}}$, as $n\to+\infty$, it follows that $\displaystyle\EE_{p,q}(\mu)\leq\lim_{n\to+\infty}E_{p,q}(u_n)=0$. On the other hand, to see the monotonicity of $\EE_{p,q}$, let $0<\mu_1<\mu_2$ and $(v_n)_n\subset H_{\mu_1}^1(\R)$ be such that $E_{p,q}(v_n)\to\EE_{p,q}(\mu_1)$ as $n\to+\infty$. Without loss of generality, assume also that $(v_n)_n\subset C_0^\infty(\R)$. Set, now, $w_n(x):=v_n(x)+u_n(x-y_n)$, where $(u_n)_n\subset H_{\mu_2-\mu_1}^1(\R)$ is defined as in the first part of the remark, but with mass $\mu_2-\mu_1$, and $y_n\in\R$ is such that the supports of $v_n$ and of $u_n(\cdot-y_n)$ are disjoint for every $n$. Then, $(w_n)_n\in H_{\mu_2}^1(\R)$ and
	\begin{align*}
	\EE_{p,q}(\mu_2)&\leq\lim_{n\to+\infty}E_{p,q}(w_n)=\lim_{n\to+\infty}E_{p,q}(v_n)+\lim_{n\to+\infty}E_{p,q}(u_n)\\
	&=\EE_{p,q}(\mu_1)+\EE_{p,q}(\mu_2-\mu_1)\leq\EE_{p,q}(\mu_1),
	\end{align*}
	where the last inequality is due to the non--positivity of $\EE_{p,q}(\cdot)$.
\end{remark}

The next two lemmas discuss the lower boundedness of $E_{p,q}$ in $H_\mu^1(\R)$ depending on the different values of $p,q$.

\begin{lemma}
	\label{lem:E=-inf}
	If $p>2$ and $q>\max\left\{4,\frac p2+1\right\}$, then $\EE_{p,q}(\mu)=-\infty$ for every $\mu>0$. Moreover, if $p\in(2,6)$ and $q=4$, then
	\[
	\EE_{p,4}(\mu)=\begin{cases}
	0 & \text{if }\mu\leq 2\\
	-\infty & \text{if }\mu>2\,. 
	\end{cases}
	\]
\end{lemma}

\begin{proof}
	Fix $p>2$, $ q>\max\left\{4,\frac p2+1\right\}$ and $\mu>0$. Setting $u(x):=\delta e^{-\delta^2|x|}$, for every $x\in \R$ and for a suitable $\delta>0$ to be chosen, we have $u\in H_1^1(\R)$. Then, letting
	\[
	v(x):=\mu^{\frac1{4-q}}u\left(\mu^{\frac{q-2}{4-q}}x\right),
	\]
	we have that $v\in H_\mu^1(\R)$ and
	\begin{equation}
	\label{eq-enumudelta}
	E_{p,q}(v)=-\mu^{\frac q{4-q}}\left(\frac{\delta^q}{q}-\frac{\delta^4}{2}\right)+\frac{2\delta^{p-2}}{p^2}\mu^{\frac{p+2-q}{4-q}}\,.
	\end{equation}
	Since $q>4$, taking $\delta$ sufficiently large yields $\displaystyle E_{p,q}(v)\leq C\left(-\mu^{\frac{q}{4-q}}+\mu^{\frac{p+2-q}{4-q}}\right)$, for a suitable constant $C>0$, and thus, by $ q>\frac p2+1$ we obtain
	\[
	\EE_{p,q}(\mu)\leq E_{p,q}(v)\to-\infty\qquad\text{as}\quad\mu\to0^+.
	\]
	This, together with Remark \ref{rem:E<0_dec}, proves that $\EE_{p,q}(\mu)=-\infty$ for every $\mu>0$.
	
	Consider, now, the case $p\in(2,6)$, $q=4$. Recall that, by \cite[Theorem 1.2]{BD21}, for every $\mu\leq 2$ and $u\in H_\mu^1(\R)$ we have
	\[
	E_{p,4}(u)>\frac12\|u'\|_2^2-\frac14|u(0)|^4\geq0.
	\]
	Hence, combining with Remark \ref{rem:E<0_dec} yields $\EE_{p,q}(\mu)=0$ for every $\mu\leq 2$. Conversely, if $\mu>2$, then, again by \cite[Theorem 1.2]{BD21}, there exists $u\in H_\mu^1(\R)$ such that
	\[
	\frac12\|u'\|_2^2-\frac14|u(0)|^4<0.
	\]
	Thus, letting $v_\alpha(x):=\sqrt{\alpha}u(\alpha\,x)$, we have $v_\alpha\in H_\mu^1(\R)$ for every $\alpha>0$ and
	\[
	\EE_{p,4}(\mu)\leq E_{p,4}(v_\alpha)=\alpha^2\left(\frac12\|u'\|_2^2-\frac14|u(0)|^4\right)+\frac{\alpha^{\frac p2-1}}p\|u\|_p^p\to-\infty\qquad\text{as}\quad\alpha\to+\infty,
	\]
	since $\displaystyle \frac p2-1<2$ whenever $p\in(2,6)$.
\end{proof}

\begin{remark}
	\label{rem:noGS}
	Incidentally, observe that the previous proof also shows that $\EE_{p,4}(\mu)$ cannot be attained when $\mu\leq2$, even though it is finite here.
\end{remark}

\begin{lemma}
	\label{lem:Ebound}
	If $p>2$ and $ q<\max\left\{4,\frac p2+1\right\}$, then $\EE_{p,q}(\mu)>-\infty$ for every $\mu>0$.
\end{lemma}

\begin{proof}
When $q<4$, finiteness of $\EE_{p,q}(\mu)$ every $\mu>0$ is a direct consequence of \eqref{eq:Gninf}. It is then left to discuss the case $p>6$ and $ 4\leq q <\frac p2+1$. Assume by contradiction that there exist $\mu>0$ and $(u_n)_n\in H^1_\mu(\R)$ such that $E_{p,q}(u_n)\to-\infty$. In particular, for every $n$,
\begin{equation}
\label{eq:assurdo}
\frac{1}{2}\|u_n^\prime\|_2^2+\frac{1}{p}\|u_n\|_p^p< \frac{1}{q}|u_n(0)|^q\,.
\end{equation}
Now, since, (again) for every $n$,
\[
\begin{split}
\left|2|u_n(0)|^q-\int_{-1}^{1}|u_n(x)|^qdx\right|&\,=\left|\int_{-1}^1 (|u_n(0)|^q-|u_n(x)|^q)dx\right| \\
&\,\leq 2q\int_{-1}^1 |u_n'(t)||u_n(t)|^{q-1}dt\le 2q\|u_n'\|_{L^2(-1,1)}\|u_n\|_{L^{2(q-1)}(-1,1)}^{q-1}\,,
\end{split}
\]
combining with Young inequality we obtain
\begin{align}
\label{proof.en.bound.1}
\frac{1}{q}|u_n(0)|^q&\,\le \frac{1}{2q}\|u_n\|_{L^q(-1,1)}^q+\|u_n'\|_{L^2(-1,1)}\|u_n\|_{L^{2(q-1)}(-1,1)}^{q-1}\nonumber\\
&\,\leq\frac{1}{2q}\|u_n\|_{L^q(-1,1)}^q+\frac{1}{6}\|u_n'\|_2^2 + \frac32\|u_n\|_{L^{2(q-1)}(-1,1)}^{2(q-1)}.
\end{align}
Since $4\leq q<\frac{p}{2}+1$, then $q<2(q-1)<p$ and by interpolation
\[
\frac32\|u_n\|_{L^{2(q-1)}(-1,1)}^{2(q-1)}\le C_1\|u_n\|_{L^q(-1,1)}^{2\theta(q-1)}\|u_n\|_{L^p(-1,1)}^{2(1-\theta)(q-1)}, 
\]
for some $C_1>0$, where 
\begin{equation}
\label{eq:theta}
\frac{1}{2(q-1)}=\frac{\theta}{q}+\frac{1-\theta}{p}\qquad\text{and}\qquad\theta\in(0,1).
\end{equation}
Since $2(1-\theta)(q-1)<p$, again by Young inequality,
\[ C_1\|u_n\|_{L^q(-1,1)}^{2\theta(q-1)}\|u_n\|_{L^p(-1,1)}^{2(1-\theta)(q-1)}\le C_2\|u_n\|_{L^q(-1,1)}^{2\theta\alpha(q-1)} + \frac{1}{3p}\|u_n\|_p^p, 
\]
with $C_2>0$ and
\[
\alpha=\frac{p}{p-2(q-1)(1-\theta)}\,,
\]
so that, by \eqref{eq:theta},
\[
\alpha(q-1)=\frac{p}{\frac{p}{q-1}-2(1-\theta)}=\frac{1}{\frac{1}{q-1}-2\frac{1-\theta}{p}}=\frac{q}{2\theta}
\]
whence $2\theta\alpha(q-1)=q$. Thus, plugging into \eqref{proof.en.bound.1}, we get
\begin{equation}\label{proof.en.bound.2}
\frac{1}{q}|u_n(0)|^q\le C_3\|u_n\|_{L^q(-1,1)}^q + \frac{1}{3p}\|u_n\|_p^p + \frac{1}{6}\|u_n'\|_2^2\,,
\end{equation}
with $C_3>0$, and combining with \eqref{eq:assurdo} yields
\begin{equation}
\label{lastineq}
|u_n(0)|^q\le C_4\|u_n\|_{L^q(-1,1)}^q\le C_5 \|u_n\|_p^q,\qquad\text{with}\quad C_4,\,C_5>0,
\end{equation}
where we also used H\"older inequality combined with the condition $q<p$. Now, on the one hand, since $E_{p,q}(u_n)\to -\infty$, then $|u_n(0)|^q\to +\infty$, so that $\|u_n\|_p\to +\infty$. On the other hand, since \eqref{eq:assurdo} implies $\|u_n\|_p^p\leq \frac{p|u_n(0)|^q}{q}$, relying again on \eqref{lastineq}, we obtain $\|u_n\|_p^p\le C_6 \|u_n\|_p^q$ for a suitable $C_6>0$. Hence, $\|u_n\|_p$ is uniformly bounded, which contradicts the previous claim.
\end{proof}

We can now complete the proof of Theorem \ref{thm:propE}.

\begin{proof}[Proof of Theorem \ref{thm:propE}]
	By Remark \ref{rem:E<0_dec} and Lemmas \ref{lem:E=-inf}--\ref{lem:Ebound}, it is left to prove the continuity of the map $\mu\mapsto\EE_{p,q}(\mu)$ on $(0,+\infty)$ when $p>2$ and $ q<\max\left\{4,\frac p2+1\right\}$. Let, then, $\mu_n\to\overline\mu$ as $n\to+\infty$, for some $\overline\mu>0$, and, for fixed $\varepsilon>0$, let $(u_n)_n\subset H_{\mu_n}^1(\R)$ be such that, for every $n$,
	\[
	\EE_{p,q}(\mu_n)\leq E_{p,q}(u_n)\leq\EE_{p,q}(\mu_n)+\varepsilon\,.
	\]
	Using \eqref{eq:GNp}--\eqref{eq:Gninf} when $q\in(2,4)$ and arguing as in the proof of Lemma \ref{lem:Ebound} when $\displaystyle 4\leq q <\frac p2+1$, one can check that $\EE_{p,q}(\mu_n)$ is uniformly bounded for large $n$ and that $(u_n)_n$ is bounded in $H^1(\R)$. Therefore, applying \eqref{eq:scalEst} with $u=u_n$, $\mu_1=\mu_n$, $\mu_2=\overline\mu$, and taking $n\to+\infty$ we obtain
	\[
	\EE_{p,q}(\overline\mu)\leq\liminf_{n\to+\infty}E_{p,q}(u_n)\leq\liminf_{n\to+\infty}\EE_{p,q}(\mu_n)+\varepsilon.
	\]
	Since $\varepsilon>0$ is arbitrary, this yields
	\begin{equation}
		\label{eq:cont1}
		\EE_{p,q}(\overline\mu)\leq\liminf_{n\to+\infty}\EE_{p,q}(\mu_n)\,.
	\end{equation}
	Similarly, relying again on \eqref{eq:scalEst} with  $\mu_1=\overline\mu$, $\mu_2=\mu_n$, for every $u\in H_{\overline\mu}^1(\R)$ we have
	\[
	\limsup_{n\to+\infty}\EE_{p,q}(\mu_n)\leq E_{p,q}(u)\,.
	\]
	Taking the infimum over $u\in H_{\overline\mu}^1(\R)$ and coupling with \eqref{eq:cont1}, we conclude.
\end{proof}

We end this section with a couple of results that are useful to address the existence of ground states discussed in Section \ref{sec:gs}. The former establishes a sufficient condition for the strict monotonicity of $\EE_{p,q}$, the latter provides a general existence criterion for ground states.

\begin{lemma}
	\label{lem:Estretdec}
	Let $\mu>0$ and $u\in H_\mu^1(\R)$ be such that $E_{p,q}(u)=\EE_{p,q}(\mu)$. If $u$ solves \eqref{nlse} for a suitable $\lambda>0$, then $\EE_{p,q}(\mu)>\EE_{p,q}(\mu_1)$, for every $\mu_1>\mu$.
\end{lemma}
\begin{proof}
	Note that, since $u$ is a ground state of $\EE_{p,q}(\mu)$, then by Lemma \ref{lem:statQual} (up to a change of sign) $u$ is a positive solution of \eqref{nlse} for some $\lambda\geq0$. Assume, in addition, that $\lambda\neq0$.

	Let, then, $v_t=u+t\varphi$ for every $t>0$, for a fixed $\varphi\in H^1(\mathbb R)$ such that $\varphi>0$. Clearly, $\displaystyle\mu_t:=\|v_t\|_2^2\downarrow\mu$ as $t\to0^+$ and, setting $f(t):=E_{p,q}(v_t)$ for every $t>0$, one has
	\[
	f'(t)=\int_\R (u' + t \varphi')\varphi' \dx + \int_\R (u+t\varphi)^{p-1}\varphi \dx - (u(0)+t\varphi(0))^{q-1}\varphi(0)\,,
	\]
	so that, by \eqref{nlse},
	\[
	\lim_{t\to0^+}f'(t)= \int_\R u^\prime \varphi^\prime dx + \int_\R u^{p-1}\varphi dx - u(0)^{q-1}\varphi(0)=-\lambda \int_\R u \varphi dx <0.
	\]
	Thus, there exists $\varepsilon>0$ such that
	\[
	\EE_{p,q}(\mu)=E_{p,q}(u)>E_{p,q}(v_t)\ge \EE_{p,q}(\mu_t),\qquad\forall t\in(0,\varepsilon).
	\]
	The result follows combining with Remark \ref{rem:E<0_dec}.
\end{proof}

\begin{lemma}
	\label{lem:critex}
	Let $q\neq\frac{p}{2}+1$ or $p<6$. If $\overline\mu>0$ is such that $-\infty<\EE_{p,q}(\overline\mu)<\EE_{p,q}(\mu)$ for every $\mu<\overline\mu$, then there exists a ground state of $E_{p,q}$ at mass $\overline{\mu}$.
\end{lemma}

\begin{proof}
	Let $u_n\in H^1_{\overline{\mu}}(\mathbb R)$ be such that $E_{p,q}(u_n)\to \mathcal{E}_{p,q}(\overline{\mu})$. Observe that, by the standard theory of rearrangements (see e.g. \cite[Section 3]{AST_CVPDE}), we can assume, without loss of generality, $u_n$ to be even and non--increasing on $\R^+$. Moreover, by the assumptions on $\EE_{p,q}(\overline\mu)$, Theorem \ref{thm:propE} ensures that $ q<\max\left\{4,\frac p2+1\right\}$. Hence, arguing as in the proof of Theorem \ref{thm:propE} we obtain that $(u_n)_n$ is bounded in $H^1(\R)$, so that, up to subsequences, $u_n\rightharpoonup u$ in $H^1(\mathbb R)$ for some $u\in H^1(\R)$. Since $u_n$ is radially decreasing, this implies $u_n\to u$ in $L^r(\mathbb R)$ for any $r\in(2,\infty]$. In particular, $u_n(0)\to u(0)$, that together with
	\[
	0\geq\EE_{p,q}(\mu)>\mathcal{E}_{p,q}(\overline{\mu})=E_{p,q}(u_n) + o(1)\ge -\frac{1}{q}|u_n(0)|^q + o(1),\qquad\text{as}\quad n\to+\infty,
	\]
	(the first inequality coming from Remark \ref{rem:E<0_dec} and the second by assumption), yields $u(0)\neq0$.
	
	Set, then, $m:=\|u\|_2^2$, so that by lower--semicontinuity $0<m\le \overline{\mu}$. We have
	\begin{multline*} \mathcal{E}_{p,q}(\overline{\mu})=\lim_{n\to+\infty}E_{p,q}(u_n)=\lim_{n\to+\infty}\left(\frac{1}{2}\|u_n'\|_2^2 + \frac{1}{p}\|u_n\|_p^p-\frac{1}{q}|u_n(0)|^q\right)\\
	\geq \liminf_{n\to+\infty} \frac{1}{2}\|u_n'\|_2^2 + \frac{1}{p}\|u\|_p^p - \frac{1}{q}|u(0)|^q  = E_{p,q}(u) + \liminf_{n\to+\infty} \frac{1}{2}\left(\|u_n^\prime\|_2^2 - \|u^\prime\|_2^2\right)\ge \mathcal{E}_{p,q}(m), 
	\end{multline*}
	the last inequality coming again by lower semicontinuity. Since, by assumption, $\EE_{p,q}(\overline\mu)<\EE_{p,q}(\mu)$ for every $\mu<\overline{\mu}$, it follows that $m=\overline{\mu}$. Hence, $u\in H_{\overline\mu}^1(\R)$ and all the above inequalities are in fact equalities, so that $E_{p,q}(u)=\EE_{p,q}(\overline\mu)$, i.e. $u$ is a ground state at mass $\overline{\mu}$.
\end{proof}

\section{Ground states: proof of Theorems \ref{thm:gs1}--\ref{thm:gs2}--\ref{thm:gs3}}
\label{sec:gs}
This section provides the proofs of the main results of the paper concerning ground states of $E_{p,q}$. We begin with Theorem \ref{thm:gs1}, namely the case
\begin{equation}
\label{eq:pq1}
p>2\qquad\text{and}\qquad 2<q<\min\left\{4,\frac{p}{2}+1\right\}\,.
\end{equation}
Before proving Theorem \ref{thm:gs1} we state the next lemma, showing that for the nonlinearity powers as in \eqref{eq:pq1} the ground state energy level is bounded uniformly in $\mu$.
\begin{lemma}
	\label{lem:limE}
	Let $p,q$ satisfy \eqref{eq:pq1}. Then
	\[
	\lim_{\mu\to+\infty}\EE_{p,q}(\mu)>-\infty\,.
	\]
\end{lemma}
\begin{proof}
	Assume by contradiction that $\displaystyle\lim_{\mu\to+\infty}\EE_{p,q}(\mu)=-\infty$ (existence of such limit is guaranteed by the monotonicity of $\EE_{p,q}$). Then, there exists a sequence $(\mu_n)_n$ such that $\mu_n\to+\infty$ and such that $\EE_{p,q}(\mu_n)<\EE_{p,q}(\mu)$ for every $\mu<\mu_n$ and for every $n$. Lemma \ref{lem:critex} thus implies that there exists $(u_n)_n$ such that $u_n\in H_{\mu_n}^1(\R)$ and $E_{p,q}(u_n)=\EE_{p,q}(\mu_n)$ for every $n$. Since $\EE_{p,q}(\mu_n)\to-\infty$, it follows that $u_n(0)\to+\infty$. However, since $u_n$ is a ground state of $E_{p,q}$ at mass $\mu_n$, then it satisfies \eqref{nlse} for some $\lambda_n\geq0$ (see Lemma \ref{lem:statQual}). Combining the boundary condition in \eqref{nlse} and \eqref{eq:emec} with the radial symmetry and positivity of $u_n$ one obtains
	\[
	\frac{u_n(0)^{2(q-1)}}{4}=\lambda_n u_n(0)^2+\frac2p u_n(0)^p.
	\]
	Thus, since $\lambda_n$ is non-negative, 
	\begin{equation}
	\label{eq-u0lim}
	u_n(0)^{p+2-2q}\leq\frac p8,
	\end{equation}
	which then contradicts $u_n(0)\to+\infty$ and concludes the proof.
\end{proof}

\begin{proof}[Proof of Theorem \ref{thm:gs1}]
	As a preliminary step, observe that $\EE_{p,q}(\mu)<0$ for every $\mu>0$. Indeed, taking e.g. $u(x)=\delta e^{-\delta^2|x|}$, for every $x\in\R$ and for a suitable $\delta>0$ to be chosen, and setting
	\[
	v(x)=\mu^{\frac1{4-q}}u\left(\mu^{\frac{q-2}{4-q}}x\right),
	\]
	it holds that $v\in H_\mu^1(\R)$ and, since $ q<\min\left\{4,\frac p2+1\right\}$ implies $0<\frac{q}{4-q}<\frac{p+2-q}{4-q}$, that $\EE_{p,q}(\mu)\leq E_{p,q}(v)<0$ as soon as $\delta$ and $\mu$ are sufficiently close to zero. This proves that $\EE_{p,q}(\mu)$ is strictly negative in a right neighborhood of $0$, and thus for every $\mu>0$ by Remark \ref{rem:E<0_dec}.
	
	Now, since $\EE_{p,q}(0)=0$ and $\EE_{p,q}(\mu)<0$ for every $\mu>0$, continuity and again Remark \ref{rem:E<0_dec} imply that either $\EE_{p,q}(\mu)$ is strictly decreasing on $[0,+\infty)$, or there exists at least one value $\widehat\mu>0$ such that $\EE_{p,q}(\widehat{\mu})<\EE_{p,q}(\mu)$ for every $\mu<\widehat{\mu}$ and $\EE_{p,q}(\widehat{\mu})=\EE_{p,q}(\mu)$ for every $\mu$ in a right neighborhood of $\widehat{\mu}$. In the latter case, by Lemma \ref{lem:critex}, there exists $u\in H_{\widehat{\mu}}^1(\R)$ such that $E_{p,q}(u)=\EE_{p,q}(\widehat{\mu})$. Moreover, by Lemma \ref{lem:Estretdec}, since $\EE_{p,q}$ is locally constant on a right neighbourhood of $\widehat{\mu}$, such $u$ must solve \eqref{nlse} with $\lambda=0$. However, by Lemma \ref{prop:l=0}, when $p\geq6$ this is impossible, whereas when $p\in(2,6)$ it can happen only if $\widehat{\mu}=\mu_0$, where $\mu_0>0$ is the mass of the unique positive solution in $H^1(\R)$ of \eqref{nls-nomass} with $\lambda=0$ defined by \eqref{eq-mu0}.
	
	As a consequence, in the case
	\[
	p\geq 6\qquad\text{and}\qquad 2<q <4,
	\]
	the ground state energy level $\EE_{p,q}$ is strictly decreasing on $(0,+\infty)$, so that ground states exist for every $\mu>0$ by Lemma \ref{lem:critex}. Moreover, since Lemma \ref{lem:mon_mu} ensures that, for this choice of $p,q$, there exists a unique positive solution of \eqref{nlse} for every $\mu>0$, it follows that (up to a change of sign) the ground state of $E_{p,q}$ is unique for every $\mu>0$. In view of Lemma \ref{lem:limE}, this completes the proof of Theorem \ref{thm:gs1}(ii).
	
	Conversely, when 
	\[
	p\in(2,6)\qquad\text{and}\qquad 2<q<\frac p2+1,
	\]
	the previous argument implies that there exists at most one value of $\mu>0$ such that $\EE_{p,q}$ is locally constant on a right neighborhood of $\mu$ and that, if such a value of the mass exists, then it coincides with $\mu_0$ defined by \eqref{eq-mu0}. Note also that, by Theorem \ref{thm:exsol} (in view of Lemmas \ref{lem:mu_as}--\ref{lem:mon_mu}), in this regime of $p,\,q$ \eqref{nlse} admits a positive solution if and only if $\mu\in(0,\mu_0]$. In particular, this implies that there exist no ground state of $E_{p,q}$ with mass larger than $\mu_0$. By Lemma \ref{lem:critex}, it entails that $\EE_{p,q}$ cannot be strictly decreasing for large masses, i.e. $\displaystyle \EE_{p,q}(\mu)=\lim_{\nu\to+\infty}\EE_{p,q}(\nu)$ for every $\mu$ large enough. However, this means that there exists at least one value of $\mu>0$ such that $\EE_{p,q}(\mu)$ is locally constant on a right neighborhood of $\mu$, so that we conclude that $\EE_{p,q}$ is strictly decreasing on $(0,\mu_0]$ and it is constant on $(\mu_0,+\infty)$. By Lemma \ref{lem:critex}, this yields existence of ground states for every $\mu\leq\mu_0$, and such ground states are unique (up to a change of sign) again by Theorem \ref{thm:exsol}. Note that, by Lemmas \ref{lem:mu_as}--\ref{lem:mon_mu}, the threshold value $\mu_0$ coincides with the value $\mu_{p,q}$ in Theorem \ref{thm:exsol}(i), which completes the proof of Theorem \ref{thm:gs1}(i). 
	
	It is thus left to show that, whenever $p>2$ and $ 2<q < \min\left\{4,\frac p2+1\right\}$, there exists $\overline\mu_{p,q}>0$ so that $\EE_{p,q}$ is concave on $(0,\overline{\mu}_{p,q}]$ and convex on $[\overline{\mu}_{p,q},+\infty)$. To this end, note first that, since we already proved that positive ground states of $E_{p,q}$ at mass $\mu$ are unique whenever they exist, the map $\mu\mapsto\lambda(\mu)$ given by
	\begin{equation}
		\label{eq:l mu}
		\lambda(\mu):=\frac{u_\mu(0)^q-\|u_\mu'\|_2^2-\|u_\mu\|_p^p}{\mu}\,,
	\end{equation}
	where $u_\mu\in H_\mu^1(\R)$ is such that $E_{p,q}(u_\mu)=\EE_{p,q}(\mu)$, is well defined on $(0,\mu_0]$ when $p\in(2,6)$ and on $(0,+\infty)$ when $p\geq 6$. By definition, $u_\mu$ solves \eqref{nlse} with $\lambda=\lambda(\mu)$. Now, we claim that, whenever $\lambda(\mu)$ is well--defined, it satisfies
	\begin{equation}
	\label{eq:derE}
	\EE_{p,q}'(\mu)=-\frac{\lambda(\mu)}2\,.
	\end{equation}
	Indeed, for $\varepsilon>0$ we have
	\begin{multline*} 
	\frac{\mathcal{E}_{p,q}(\mu+\varepsilon)-\mathcal{E}_{p,q}(\mu)}{\varepsilon}\le \frac{E_{p,q}\left(\sqrt{\frac{\mu+\varepsilon}{\mu}}u_\mu\right)-E_{p,q}(u_\mu)}{\varepsilon} \\[.2cm]
	= \frac{\frac{1}{2}\left(\frac{\mu+\varepsilon}{\mu}-1\right)\|u_\mu'\|_2^2 + \frac{1}{p}\left(\left(\frac{\mu+\varepsilon}{\mu}\right)^{p/2}-1\right)\|u_\mu\|_p^p - \frac{1}{q}\left(\left(\frac{\mu+\varepsilon}{\mu}\right)^{q/2}-1\right)u_\mu(0)^q}{\varepsilon}\\[.2cm]
	=\frac{\frac{\|u_\mu'\|_2^2+\|u_\mu\|_p^p-u_\mu(0)^q}{2\mu}\varepsilon+o(\varepsilon)}{\varepsilon} =-\frac{\lambda(\mu)}{2}+o(1),\qquad\text{as}\quad\varepsilon\to0^+,
	\end{multline*}
	so that
	\[
	 \limsup_{\varepsilon\to0^+}\frac{\mathcal{E}_{p,q}(\mu+\varepsilon)-\mathcal{E}_{p,q}(\mu)}{\varepsilon}\le-\frac{\lambda(\mu)}{2}\,.
	\]
	Similarly,
	\begin{multline*} 
	\frac{\mathcal{E}_{p,q}(\mu+\varepsilon)-\mathcal{E}_{p,q}(\mu)}{\varepsilon}\ge \frac{E_{p,q}(u_{\mu+\varepsilon}) - E\left(\sqrt{\frac{\mu}{\mu+\varepsilon}}u_{\mu+\varepsilon}\right)}{\varepsilon}\\
	= \frac{\frac{1}{2}\left(1-\frac{\mu}{\mu+\varepsilon}\right)\|u_{\mu+\varepsilon}'\|_2^2 + \frac{1}{p}\left(1-\left(\frac{\mu}{\mu+\varepsilon}\right)^{p/2}\right)\|u_{\mu+\varepsilon}\|_p^p - \frac{1}{q}\left(1-\left(\frac{\mu}{\mu+\varepsilon}\right)^{q/2}\right)u_{\mu+\varepsilon}(0)^q}{\varepsilon}\\
	=\frac{\frac{\|u_{\mu+\varepsilon}'\|_2^2+\|u_{\mu+\varepsilon}\|_p^p-u_{\mu+\varepsilon}(0)^q}{2(\mu+\varepsilon)}\varepsilon+o(\varepsilon)}{\varepsilon} =-\frac{\lambda(\mu)}{2}+o(1)\qquad\text{as}\quad\varepsilon\to0^+\,,
	\end{multline*}
	where we exploited the strong continuity in $H^1(\R)$ of $u_{\mu+\varepsilon}$ as $\varepsilon\to0$ given by Remark \ref{rem:cont_sol}. Passing to the liminf as $\varepsilon\to0^+$ and coupling the previous estimates yield
	\[ \lim_{\varepsilon\to0^+}\frac{\mathcal{E}_{p,q}(\mu+\varepsilon)-\mathcal{E}_{p,q}(\mu)}{\varepsilon}=-\frac{\lambda(\mu)}{2}\,.
	\]
	Since the analogous computations can be done for $\varepsilon\to0^-$, we obtain \eqref{eq:derE}.
	
	In view of \eqref{eq:derE}, the concavity/convexity properties of $\EE_{p,q}$ can be discussed through the monotonicity properties of $\lambda(\mu)$. To this aim, recall that, since $u_\mu$ solves \eqref{nlse} with $\lambda=\lambda(\mu)$, it corresponds to the unique value $t\in(1,+\infty]$ given by \eqref{eq:condt} (recall that $t=+\infty$ represents, when present, the solution with $\lambda=0$, i.e. $\mu=\mu_0$). Thus $\lambda(\mu)$ can be equivalently seen as a function $\lambda(t)$ for $t\in(1,+\infty]$ (it is sufficient to invert $g$ in \eqref{eq:condt}) and, by the proof of Lemma \ref{prop:l>0}, when $ q<\frac p2+1$ the map $t\mapsto\lambda(t)$ has been shown to be increasing on (1,$\overline{t})$ and decreasing on $(\overline{t},+\infty)$, for some $\overline{t}>1$. Moreover, Lemma \ref{lem:mon_mu} ensures that the mass $\mu$ of $u_\mu$ can be rewritten as a function $\mu(t)$ that is strictly increasing on its domain for $p,\,q$ satisfying \eqref{eq:pq1}. Hence, the inverse function $t(\mu)$ is well-defined and strictly increasing on its domain, in turn implying that there exists $\overline{\mu}_{p,q}>0$ such that the map $\lambda(\mu)=\lambda(t(\mu))$ is increasing on $(0,\overline{\mu}_{p,q})$ and decreasing for $\mu>\overline\mu_{p,q}$.
\end{proof}

\begin{remark}
	Note that, when $p\in(2,6)$, the final part of the argument for the proof of the concavity/convexity features of $\EE_{p,q}$ in Theorem \ref{thm:gs1} only applies, in fact, to masses in the interval $[0,\mu_{p,q}]$. However, as $\EE_{p,q}$ is constant for $\mu\geq\mu_{p,q}$, this is enough to conclude.
\end{remark}

Let us focus, now, to Theorem \ref{thm:gs2}, where
\begin{equation}
	\label{eq:pq2}
	p\in(2,6)\quad\text{and}\quad\frac p2+1<q<4,\qquad\text{or}\qquad p\geq6\quad\text{and}\quad4\leq q <\frac p2+1.
\end{equation}

\begin{proof}[Proof of Theorem \ref{thm:gs2}]
	Preliminarily, set
	\[
	\widetilde{\mu}_{p,q}:=\sup\left\{\mu>0\,:\,\EE_{p,q}(\mu)=0\right\}\,.
	\]
	Observe first that, for $p,\,q$ as in \eqref{eq:pq2}, we have $\widetilde{\mu}_{p,q}>0$. Indeed, if this were not the case, there would exists a sequence $(\mu_n)_n$ such that $\mu_n\downarrow0$ and $\EE_{p,q}(\mu_n)<\EE_{p,q}(\mu)\leq0$ for every $\mu<\mu_n$ and for every $n$. By Lemma \ref{lem:critex}, this would yield the existence of a sequence of ground states $(u_n)_n$ at mass $\mu_n$, so that $(u_n)_n$ would be a sequence of solutions of \eqref{nlse} with mass $\mu_n\to0$. Since this is impossible by Theorem \ref{thm:exsol}(iii)-(v), it follows that $\widetilde{\mu}_{p,q}>0$.
	
	Let us then show that $\widetilde{\mu}_{p,q}<+\infty$. To this end, take $u(x)=\delta e^{-\delta^2|x|}$ and $v(x)=\mu^\alpha u(\mu^{2\alpha-1}x)$, for suitable $\alpha$, $\delta$ to be chosen, so that $v\in H_\mu^1(\R)$ and 
	\[
	E_{p,q}(v)=\frac{\delta^4}{2}\mu^{4\alpha-1}+\frac2{p^2}\delta^{p-2}\mu^{\alpha(p-2)+1}-\frac{\delta^q}{q}\mu^{\alpha q}\,.
	\]
	It is then easy to see that, for $p,q$ as in \eqref{eq:pq2}, there always exist a choice of $\alpha$ and $\delta$ for which $E_{p,q}(v)<0$ for a fixed $\mu$ large enough. Indeed, it is enough to take
	\[
	\left\{
	\begin{array}{ll}
	\displaystyle \alpha=\frac1{4-q},\quad\delta>0\quad\text{small enough},  & \quad\text{if}\quad p\in(2,6),\:\frac p2+1 < q < 4,\\[.4cm]
	\displaystyle-\frac1{q-4}< \alpha <- \frac1{p-2-q},\quad\delta=1,  & \quad\text{if}\quad p>6,\:4< q <\frac p2+1,\\[.4cm]
	\displaystyle\alpha<-\frac1{p-6},\quad\delta=1,  & \quad\text{if}\quad p>6,\:q=4.
	\end{array}
	\right.
	\]
	Moreover, this also gives $\displaystyle\lim_{\mu\to+\infty}\EE_{p,q}(\mu)=-\infty$ when $p\in(2,6)$ and $ \frac p2+1<q<4$. Conversely, when $p>6$ and $ 4\leq q <\frac p2+1$, arguing exactly as in the proof of Lemma \ref{lem:limE} (note that to make the crucial estimate \eqref{eq-u0lim} significant it sufficies that $q<\frac{p}{2}+1$) one obtains $\displaystyle\lim_{\mu\to+\infty}\EE_{p,q}(\mu)>-\infty$. 
	
	Now, we prove that $\EE_{p,q}$ is strictly decreasing on $[\widetilde{\mu}_{p,q},+\infty)$. When $p>6$, this is immediate because, if $\EE_{p,q}$ were locally constant on a right neighbourhood of any $\mu>\widetilde{\mu}_{p,q}$, then by Lemmas \ref{lem:Estretdec}--\ref{lem:critex} there would exist a ground state of $E_{p,q}$ solving \eqref{nlse} with $\lambda=0$, but this is impossible by Theorem \ref{thm:nomass}(ii). Conversely, if $p\in(2,6)$ and $ \frac p2+1<q < 4$, then, taking $\widetilde{\mu}_{p,q}<\mu_1<\mu_2$ and using \eqref{eq:scalEst} with $\alpha=\frac1{4-q}$ yield
	\begin{equation}
	\label{eq:Emu12}
	\EE_{p,q}(\mu_2)<\left(\frac{\mu_2}{\mu_1}\right)^{\frac q{4-q}}E_{p,q}(u),\qquad\forall u\in H_{\mu_1}^1(\R)\,,
	\end{equation}
	and passing to the infimum over $u\in H_{\mu_1}^1(\R)$ (and using that $\EE_{p,q}(\mu_1)<0$) gives $\EE_{p,q}(\mu_2)<\EE_{p,q}(\mu_1)$ as claimed.
	
	Hence, all the properties of $\EE_{p,q}$ listed in Theorem \ref{thm:gs2} are proved. Moreover, by Lemma \ref{lem:critex}, the strict monotonicity of $\EE_{p,q}$ on $[\widetilde{\mu}_{p,q},+\infty)$ implies that ground states exist for every $\mu>\widetilde{\mu}_{p,q}$. On the contrary, no ground state exists when $\mu<\widetilde{\mu}_{p,q}$, because if $u\in H_\mu^1(\R)$ were such a ground state, then it would be $E_{p,q}(u)=0$ that, together with \eqref{eq:Emu12}, would imply  $\EE_{p,q}(\overline{\mu})<0$ for some $\overline{\mu}\in(\mu,\widetilde{\mu}_{p,q})$, contradicting the definition of $\widetilde{\mu}_{p,q}$.
	
	Moreover, for $q\neq4$, ground states exist also when $\mu=\widetilde{\mu}_{p,q}$. This is a direct consequence of Remark \ref{rem:cont_sol} and Theorem \ref{thm:exsol}. Indeed, let $\mu_n\downarrow\widetilde{\mu}_{p,q}$ and $u_n\in H_{\mu_n}^1(\R)$ be such that $E_{p,q}(u_n)=\EE_{p,q}(\mu_n)$ for every $n$. Then, $u_n$ solves \eqref{nlse} and, by Remark \ref{rem:cont_sol}, converges strongly in $H^1(\R)$ to a solution $u\in H_{\widetilde{\mu}_{p,q}}^1(\R)$ of \eqref{nlse} such that $\displaystyle E_{p,q}(u)=\lim_{n\to+\infty}E_{p,q}(u_n)=\lim_{n\to+\infty}\EE_{p,q}(\mu_n)=0$, i.e. $u$ is a ground state of $E_{p,q}$ in $H_{\widetilde{\mu}_{p,q}}^1(\R)$. Note that here the actual value of $\widetilde{\mu}_{p,q}$ does not play any role. Since Theorem \ref{thm:exsol}(iii) guarantees that \eqref{nlse} admits a solution if and only if $\mu\geq \mu_{p,q}>0$, we only know that $\widetilde{\mu}_{p,q}\geq \mu_{p,q}$. 
	
	We are thus left to discuss the case $q=4$ and $\mu=\widetilde{\mu}_{p,4}$. Note that, by Theorem \ref{thm:exsol}(v), in order to prove that in this case ground states do not exists at the critical value of the mass it is enough to show that $\widetilde{\mu}_{p,4}=2$. Actually, in view of the previous results, it is sufficient to establish $\widetilde{\mu}_{p,4}\leq2$, since $\widetilde{\mu}_{p,4}<2$ is already ruled out again by Theorem \ref{thm:exsol}(v). We do this by showing that $\EE_{p,4}(\mu)<0$ for every $\mu>2$. Indeed, by \cite[Theorem 1.2]{BD21}, for every $\mu>2$ there exists $u\in H_\mu^1(\R)$ such that
	\begin{equation}
	\label{eq:Du<0}
	\frac12\|u'\|_2^2-\frac14|u(0)|^4<0.
	\end{equation}
	Given $\sigma>0$, set then $u_\sigma(x):=\sqrt{\sigma}u(\sigma\,x)$ for every $x\in \R$, so that $u_\sigma\in H_\mu^1(\R)$ and
	\[
	E_{p,4}(u_\sigma)=\sigma^2\left(\frac12\|u'\|_2^2-\frac14|u(0)|^4\right)+\frac{\sigma^{\frac p2-1}}p\|u\|_p^p\,,
	\]	
	that by \eqref{eq:Du<0} and $p>6$ yields $E_{p,4}(u_\sigma)<0$ as soon as $\sigma$ is sufficiently close to zero.
	
	To conclude, observe that the uniqueness (up to a change of sign) of ground states of $E_{p,4}$ at mass $\mu>2$ follows by the uniqueness results for solutions of \eqref{nlse} given again by Theorem \ref{thm:exsol}.
\end{proof}

\begin{proof}[Proof of Theorem \ref{thm:gs3}]
	Let $p>2$, $ q=\frac p2+1$, and $\mu>0$. By Theorem \ref{thm:exsol}(vi) \eqref{nlse} has no solution if $p\leq8$. Let then $p>8$. In this setting, again by Theorem \ref{thm:exsol}(vi), \eqref{nlse} admits a unique positive solution. Since by Theorem \ref{thm:propE} it is always true that $\EE_{p,q}(\mu)\leq0$, to prove that $\EE_{p,q}(\mu)$ is never attained it is enough to show that the energy of such solution is always strictly positive. To do this, recall that, by the proof of Lemma \ref{prop:l>0}, each solution in $H^1(\R)$ of \eqref{nls-nomass} corresponds to one and only one value $t\in(1+\infty)$ via \eqref{eq:condt}. Denoting by $u_t$ the solution corresponding to $t\in(1,+\infty)$, we now exploit the explicit formula \eqref{eq:exp_ul>0} to compute $E_{p,q}(u_t)$. Observe first that a direct computation yields
	\[
	\|u_t'\|_2^2= \frac{2^\frac{2p-6}{p-2}\lambda^\frac{p+2}{2(p-2)}p^\frac{2}{p-2}}{p-2}\int_1^t s^2(s^2-1)^\frac{4-p}{p-2}ds\,. 
	\]
	Since
	\[
	\int_1^t s^2(s^2-1)^\frac{4-p}{p-2}ds=\int_1^t(s^2-1)^\frac{2}{p-2}ds + \int_1^t(s^2-1)^\frac{4-p}{p-2}ds,
	\]
	and an integration by parts yields
	\[
	\int_1^t (s^2-1)^\frac{2}{p-2}ds=t(t^2-1)^\frac{2}{p-2} - \frac{4}{p-2}\int_1^t s^2(s^2-1)^\frac{4-p}{p-2}ds,
	\]
	one has
	\begin{equation}
	\label{int}
	\int_1^t s^2(s^2-1)^\frac{4-p}{p-2}ds=\frac{p-2}{p+2}\left[t(t^2-1)^\frac{2}{p-2} + I(t)\right],
	\end{equation}
	where $I(t)$ is the function defined in \eqref{eq:It}. Hence, by \eqref{int},
	\[ \frac{1}{2}\|u_t'\|_2^2=\frac{2^\frac{p-4}{p-2}p^\frac{2}{p-2}\lambda^\frac{p+2}{2(p-2)}}{p+2}\left(t(t^2-1)^\frac{2}{p-2}+I(t)\right).
	\]
	Similarly, using again \eqref{eq:exp_ul>0} and \eqref{int} we obtain
	\[ \frac{1}{p}\|u_t\|_p^p=\frac{2^\frac{p-4}{p-2}\lambda^\frac{p+2}{2(p-2)}p^\frac{2}{p-2}}{p-2}\int_1^t(s^2-1)^\frac{2}{p-2}ds=\frac{2^\frac{p-4}{p-2}\lambda^\frac{p+2}{2(p-2)}p^\frac{2}{p-2}}{p+2}\left(t(t^2-1)^\frac{2}{p-2}-\frac{4}{p-2}I(t)\right),
	\]
	and  
	\[ \frac{1}{q}|u_t(0)|^q=\frac{1}{q}\left(\frac{p\lambda}{2}\right)^\frac{q}{p-2}(t^2-1)^\frac{q}{p-2}\,,
	\] 
	so that
	\[
	\label{en.for.1}
	E_{p,q}(u_t)=\lambda^\frac{q}{p-2}(p(t^2-1))^\frac{2}{p-2}
	\left(\frac{2^\frac{p-4}{p-2}\lambda^\frac{p+2-2q}{2(p-2)}}{p+2}\left(2t +\frac{p-6}{p-2}\frac{I(t)}{(t^2-1)^\frac{2}{p-2}}\right)-\frac{p^\frac{q-2}{p-2}}{q2^\frac{q}{p-2}}(t^2-1)^\frac{q-2}{p-2}\right).
	\]
	Taking $q=\frac p2+1$ and recalling \eqref{eq:fgdiag}, the previous formula becomes
	\[
	E_{p,q}(u_t)=\frac{2^\frac{p-4}{p-2}p^\frac{2}{p-2}\lambda^\frac{p+2}{2(p-2)}(t^2-1)^{\frac{2}{p-2}}}{(p+2)(p-2)}I(t)>0
	\]
	and we conclude.
\end{proof}

\section*{Statements and Declarations}

\noindent\textbf{Conflict of interest.}  The authors declare that they have  no conflict of interest.
\bigskip 

\noindent\textbf{Data availability statement.} Data sharing not applicable to this article as no dataset was generated or analysed during the current study.

\bigskip

\noindent\textbf{Acknowledgements.} D.B, S.D. and L.T. acknowledge that this study was carried out within the project E53D23005450006 ``Nonlinear dispersive equations in presence of singularities'' -- funded by European Union -- Next Generation EU within the PRIN 2022 program (D.D. 104 - 02/02/2022 Ministero dell'Universit\`a e della Ricerca). This manuscript reflects only the author's views and opinions and the Ministry cannot be considered responsible for them.

\end{document}